\documentclass[a4paper]{scrartcl}
\usepackage[]{amsmath,amssymb}
\usepackage{amsthm}
\usepackage{enumitem}
\usepackage[utf8]{inputenc}
\usepackage[T1]{fontenc}
\usepackage{lmodern}
\usepackage{fourier}
\usepackage{graphicx}
\DeclareMathOperator{\oD}{d}
\def\D{\oD\!}
\def\dif{\,\D}

\usepackage{tikz}
\def\R{{\mathbf R}}

\def\N{{\mathbf N}}

\def\car#1{{\mathbf 1}}

\newcommand{\esp}[1]{{\mathbf E}\left[#1\right]}
\newcommand{\espi}[2]{{\mathbf E}_{#1}\left[{#2}\right]}

\def\P{\mathbf P}
\def\Q{\mathbf Q}
\def\F{{\mathcal F}}
\def\/{\,|\,} 

\def\NN{\mathfrak N} 
\def\Id{\operatorname{Id}}
\def\indic{\mathbf 1}
\def\L{\, |\,}

\def\car{\mathbf 1}

\def\pred{\mathcal{P}}
\def\m{\mathfrak m}

\def\L{\operatorname{L}}
\def\Yfrak{\mathfrak Y}
\def\Zfrak{\mathfrak Z}
\def\MM{\mathfrak M^1}
\def\pc{\operatorname{pc}}

\def\Ent{H}

\def\hawkes{\text{g-H}_y}

  \def\Ymap{\mathbf Y}
  \def\Zmap{\mathbf Z}
  
 \definecolor{fonce}{HTML}{415458}
 \definecolor{cadet}{HTML}{5D737E}
 \definecolor{vertdeau}{HTML}{7E795D}
 \definecolor{ocre}{HTML}{58414F}
 \definecolor{creme}{HTML}{7E5D72}

\RequirePackage[framemethod=default]{mdframed}
\newmdenv[skipabove=7pt,
skipbelow=7pt,
rightline=false,
leftline=true,
topline=false,
bottomline=false,
backgroundcolor=ocre!10,
linecolor=ocre,
innerleftmargin=5pt,
innerrightmargin=20pt,
innertopmargin=5pt,
innerbottommargin=5pt,
leftmargin=0cm,
rightmargin=0cm,
linewidth=4pt]{eBox}

\makeatletter
\newtheoremstyle{ocrenumbox}
{0pt}
{0pt}
{\itshape}
{}
{\small\bfseries\sffamily\color{cadet}}
{\;}
{0.25em}
{\small\sffamily\color{ocre}\thmname{#1}\nobreakspace\thmnumber{\@ifnotempty{#1}{}\@upn{#2}}
\thmnote{\nobreakspace\the\thm@notefont\sffamily\bfseries\color{black}---\nobreakspace#3.}} 

\newtheoremstyle{blacknumex}
{5pt}
{5pt}
{\normalfont}
{} 
{\small\bfseries\sffamily}
{\;}
{0.25em}
{\small\sffamily{\tiny\ensuremath{\blacksquare}}\nobreakspace\thmname{#1}\nobreakspace\thmnumber{\@ifnotempty{#1}{}\@upn{#2}}
\thmnote{\nobreakspace\the\thm@notefont\sffamily\bfseries---\nobreakspace#3.}}

\newtheoremstyle{blacknumbox} 
{0pt}
{0pt}
{\small}
{}
{\small\bfseries\sffamily}
{~}
{0.25em}
{\small\sffamily\thmname{#1}\nobreakspace\thmnumber{\@ifnotempty{#1}{}\@upn{#2}}
\thmnote{\nobreakspace\the\thm@notefont\sffamily\bfseries---\nobreakspace#3.}}

\newtheoremstyle{ocrenum}
{5pt}
{5pt}
{\normalfont}
{}
{\small\bfseries\sffamily\color{ocre}}
{\;}
{0.25em}
{\small\sffamily\color{ocre}\thmname{#1}\nobreakspace\thmnumber{\@ifnotempty{#1}{}\@upn{#2}}
\thmnote{\nobreakspace\the\thm@notefont\sffamily\bfseries\color{black}---\nobreakspace#3.}} 
\makeatother

\newmdenv[skipabove=7pt,
skipbelow=7pt,
rightline=false,
leftline=true,
topline=false,
backgroundcolor=ocre!10,
bottomline=false,
linecolor=ocre,
innerleftmargin=5pt,
innerrightmargin=20pt,
innertopmargin=5pt,
innerbottommargin=5pt,
leftmargin=0cm,
rightmargin=0cm,
linewidth=4pt]{dBox}

\newmdenv[skipabove=7pt,
skipbelow=7pt,
skipbelow=7pt,
backgroundcolor=fonce!10,
rightline=false,
leftline=true,
bottomline=false,
topline=false,
linewidth=4pt,
linecolor=fonce,
innerleftmargin=5pt,
innerrightmargin=20pt,
innertopmargin=5pt,
leftmargin=0cm,
rightmargin=0cm,
innerbottommargin=5pt]{tBox}

\theoremstyle{ocrenumbox}
\newtheorem{theoremeT}{Theorem}
\newtheorem{lemmaT}[theoremeT]{Lemma}
\newtheorem{corollaryT}[theoremeT]{Corollary}
\theoremstyle{blacknumex}
\newtheorem{exampleT}{Example}
\newtheorem{remarkT}{Remark}
\theoremstyle{blacknumbox}
\newtheorem{definitionT}{Definition}
\newenvironment{theorem}{\begin{tBox}\begin{theoremeT}}{\end{theoremeT}\end{tBox}}
\newenvironment{lemma}{\begin{tBox}\begin{lemmaT}}{\end{lemmaT}\end{tBox}}
\newenvironment{definition}{\begin{dBox}\begin{definitionT}}{\end{definitionT}\end{dBox}}

\newenvironment{remark}{\begin{remarkT}}{\hfill{\tiny\ensuremath{\blacksquare}}\end{remarkT}}
\newenvironment{corollary}{\begin{tBox}\begin{corollaryT}}{\end{corollaryT}\end{tBox}}
\numberwithin{theoremeT}{section}
\numberwithin{definitionT}{section}
\def\ZZ{\Upsilon}

\title{Invertibility of functionals of the Poisson process and applications}
\author{L. Coutin \and L. Decreusefond}
\date{2022}
\begin{document}
\maketitle{}
\begin{abstract}
  We show that
  solving SDEs with constant volatility on the Wiener space is the analog of
  constructing  Hawkes-like processes, i.e. self excited point process, on the
  Poisson space. Actually,  both problems are linked to the invertibility of some
  transformations on the sample paths which respect absolute continuity: adding
  an adapted drift for the Wiener space,
  making a random time change for the Poisson space.  Following previous
  investigations by Üstünel on the Wiener space, we establish an entropic criterion on the Poisson
  space which ensures the invertibility of such a transformation. As a
  consequence of this criterion, we improve the variational representation of
  the entropy with respect to the Poisson process distribution. Pursuing the
  Wiener-Poisson analogy so established, we define several notions  of
  generalized Hawkes processes as weak or strong solutions of some fixed point
  equations and show  a Yamada-Watanabe like theorem for these new equations. As
  a consequence, we find another construction of the classical (even non linear)
  Hawkes processes without the recourse to a Poisson measure.
\end{abstract}

\medskip
\noindent{Math Subject Classification: 60G55}
\setcounter{tocdepth}{1}
\tableofcontents
\def\veta{\omega}
\section{Introduction}
\label{sec:motiv}
The simplest non linear stochastic differential equations are those of the form:
\begin{equation}\label{eq_hawkes_brouillon:10}
  X(0)=0 \text{ and} \dif X(t)=\dot u(X(t))\dif t +\dif B(t)\Longleftrightarrow X(t)=\int_{0}^{t}\dot u(X(s))\dif s+B(t),
\end{equation}
where $B$ is a standard Brownian motion. We call these equations where the
coefficient in front of $B$ is equal to $1$, \emph{volatility-$1$} Brownian
SDEs. Denote by $W$ the space of continuous functions over $[0,1]$, null at time
$0$, equipped with the Wiener measure~$\mu$. We may view the process $X$ as a
map from $W$ into itself by the construction:
\begin{align*}
  X\, :\, W & \longrightarrow W                             \\
  \omega    & \longmapsto \Bigl(t\mapsto X(\omega,t)\Bigr).
\end{align*}
Consider also the map:
\begin{equation}\label{eq_hawkes_brouillon:8}
  \begin{split}
    U\, :\, W & \longrightarrow W                                                                 \\
    \omega    & \longmapsto \left(t\mapsto \omega(t)-\int_{0}^{t} \dot u(\omega(s))\dif s\right).
  \end{split}
\end{equation}
Then, \eqref{eq_hawkes_brouillon:10} is equivalent to say that $X$ is a solution
of the equation
\begin{equation*}
  U\circ X=\Id_{W}, \mu \text{-a.s.}.
\end{equation*}
Hence, to solve \eqref{eq_hawkes_brouillon:10} is to invert $U$. In
\cite{ustunel_entropy_2009} and further on in
\cite{hartmann2016variational,LassalleStochasticinvertibilityrelated2012,Ustunel:2014aa},
the authors showed that $U$ is invertible if and only if
\begin{equation}\label{eq_hawkes_brouillon:38}
  \Ent(U^{\#}\mu\,|\,\mu)=\frac12\,\esp{\int_{0}^{1}\dot u(\omega(s))^{2}\dif s}
\end{equation}
where $\Ent(U^{\#}\mu\,|\,\mu)$ is the relative entropy between $U^{\#}\mu$, the
image measure of $\mu $ by $U$, and $\mu$ itself.

Let $\mathbb{D}$ be the Skorohod space of right continuous with left limits
(rcll for short) functions equipped with the
Poisson measure $\pi$, i.e. the law of a unit rate Poisson process on $\R^{+}$.
The initial interpretation of the Poisson counterpart of
\eqref{eq_hawkes_brouillon:10} is to solve
\begin{equation}
  \label{eq_hawkes_core:21}
  Y(t)=\int_{0}^{t}\dot u(Y(s))\dif s +N(t),
\end{equation}
as the inversion of the map $\tilde U$, defined formally as $U$ but from
$\mathbb{D}$ into itself. That is to say that we seek a rcll process $Y$ which
satisfies $\tilde U\circ Y=\Id_{\mathbb{D}},\ \pi-\text{a.s.}$
This appears to be inconsistent with the methods of the previously mentioned
works.
Actually, to grasp the difference between the Wiener and Poisson settings and
consequently to appreciate the inherent differences between equations
\eqref{eq_hawkes_brouillon:10} and \eqref{eq_hawkes_core:21}, we must go back to
the basics of the Girsanov theorem. A measure $\nu$ on $W$ is absolutely
continuous with respect to~$\mu$ if there exists an adapted process $u$ in the
Cameron-Martin space such $B-u$ is a $\nu$-local martingale of square bracket
$(t\mapsto t)$. Since the square bracket of a semi-martingale is unchanged by
the addition of a finite variation process, the Lévy's characterization theorem
says that $B-u$ is a $\nu$ Brownian motion. It is this quasi-invariance which is
used to prove that SDEs like \eqref{eq_hawkes_brouillon:10}, have weak solutions
under mild assumptions on $\dot u$ and which is the key to the investigations
about invertibility.

Now, the main obstacle to a direct generalisation of this approach to the
Poisson space  is that the Girsanov theorem for the Poisson process is best
expressed in terms of point processes rather than in terms of rcll functions as
it involves the notion of compensator which is specific to the former. Let
$\NN$ be the space of configurations on $[0,+\infty)$ (see the exact definition
below) and $\pi$ the probability on $\NN$ such that the canonical process $N$ is
a Poisson process of unit intensity. The Girsanov theorem says that if $\nu$ is
absolutely continuous with respect to $\pi$, there exists a non-decreasing
predictable process, denoted by $y$, such that $N-y$ is a $\nu$ local
martingale. The difference here is that the $\nu$-compensator of $N$ is $y$ so
neither $N$ nor $N-y$ (which is not even a point process) are Poisson processes
under~$\nu$. Otherwise stated,  the
transformation of sample paths induced by an absolutely continuous change of
probability in the Poisson framework is no longer a translation, i.e. the
addition to the nominal path of a regular function.

We have thus to construct another transformation of the sample paths of $N$ such
that the process we obtain, after a change of probability measure, is still a
$\nu$-Poisson process of unit intensity. This question has been seldom addressed
(see
\cite{BassMalliavinCalculusPure1986,DecreusefondPerturbationAnalysisMalliavin1998})
and not in a form which is convenient for our present goal. It turns out that it
is the process $N(y^{*}(t))$, where $y^{*}$ is the right-inverse of $y$, which
plays the role of $B-u$ in the Poisson space. Thus, the true analog of the map
$U$ defined in \eqref{eq_hawkes_brouillon:8} is the map~$\Ymap$ defined as
\begin{equation*}
  \begin{split}
    \Ymap\, :\, \NN & \longrightarrow \NN       \\
    N               & \longmapsto Y:=N\circ y^{*}
  \end{split}
\end{equation*}
and not the map~$\tilde U$. Consequently, the true analog of the invertibility of
$U$  is to invert $\Ymap$ in the space of configurations. We show in  Lemma
\ref{lem:ConditionInversibilite} that this amounts to find a point process $Z$ with compensator $z$  such that
\begin{equation}
\label{eq_hawkes_core:20}z^{*}(N,\ y ^{*}(Z,t))=t, \text{ for any } t\ge 0.
\end{equation}
The interpretation of this equation is the following.
Given a point process $Y$ on the half-line which is adapted to a filtration
$\mathcal F$, we can always construct a $\mathcal F$-predictable process $y$,
known as its compensator, such that $Y-y$ is a ${\cal F}$-local martingale (see
\cite{Jacod1975}). The reverse question which is, given a $\mathcal{F}$-predictable, non-decreasing,
right-continuous, null at time $0$, process $y$, to devise the existence of a
 point process $Y$ such that $y$ is the $\mathcal{F}_{Y}$-compensator of $Y$ has, to the best of our knowledge, never been
addressed in full generality. The only situation we are aware of, where we only have
a partial  solution, is related to
the notion of Hawkes processes. Recall that, given two deterministic functions
$\phi$ and $\psi$, a Hawkes process \cite{hawkes74,massoulie96} is a point
process $H$ such that
$H(t)-\int_{0}^{t}\psi\left(\lambda+\int_{0}^{s}\phi(s-r)\dif H(r)\right)\dif s$
is a local martingale. The usual way to proceed is to construct $H$ as the
solution of a differential equation driven by a marked Poisson process $\Phi$ on
$\R^{+}\times \R^{+}$. This means that $H$  and its  compensator
\begin{equation}
\label{eq_hawkes_core:29}y(t)=\int_{0}^{t}\psi\left(\lambda+\int_{0}^{s}\phi(s-r)\dif H(r)\right)\dif s
\end{equation}
are adapted with respect to the $\sigma$-field generated by $\Phi$ and not to the
minimal $\sigma$-field we could hope for, which is the one generated by $H$
itself. We show below that, if $N$ is a unit rate Poisson process of
$[0,+\infty)$, the map $\Ymap$ is right invertible (i.e.
\eqref{eq_hawkes_core:20} is satisfied for some point process $Z$) if the  point
process $Z$  is the solution of the equation
\begin{equation}
  \label{eq_hawkes_core:22}Z(t)=N\Bigl(y(Z,t)\Bigr),
\end{equation}
see Theorem~\ref{thm_hawkes_brouillon:HawkesEquivautInverseDroite} for a precise
statement. The form of \eqref{eq_hawkes_core:22} entails that $y(Z,t)$ is the
compensator of $Z$ and Corollary~\ref{lem_hawkes_core:egalité_tribu_hawkes}
ensures that it is adapted to the minimal filtration generated by~$Z$. The process
$Z$ is what we call a generalized Hawkes process (g-Hawkes for short) as our
results do not depend on a particular expression of $y$ as in~\eqref{eq_hawkes_core:29}. This
means that constructing a g-Hawkes process is the Poissonian analogue to solving
volatility-1 Brownian SDEs, like \eqref{eq_hawkes_brouillon:10}, in the Wiener
space. This point of view, which, to the best of our knowledge, is new, has
major consequences as we can now transfer the problems known on SDEs (weak,
strong and martingales solutions, perturbations, stationarity, etc.) to g-Hawkes
processes. We only focus here on the different notions of solutions for the
g-Hawkes problem and show the analogue of the Yamada-Watanabe theorem (see
Theorem~\ref{thm_hawkes_core:YW}). We can then construct a classical, possibly
with a non linear intensity, Hawkes process without the recourse to a larger
probability space as usually done (\cite{Delattre2016,massoulie96,hawkes74} and
references therein).

The variational representation of entropy is a crucial theorem of the theory of
large deviations \cite{Deuschel2001}, also known as the Boué-Dupuis or Borell
formula \cite{Borell1997,10.1214/aop/1022855876} for the Gaussian measure on
$\R^{n}$. It has been extended to the Wiener space by Léhec in \cite{Lehec2013}
and with weaker hypothesis by Üstünel in \cite{ustunel_entropy_2009}, as a
consequence of the entropic criterion for the invertibility of $U$, see
\eqref{eq_hawkes_brouillon:38}. The core of the proof involves inverting a map
akin to $U$. Without the entropic criterion, the only known alternative approach
is to solve equation \eqref{eq_hawkes_brouillon:10}, which imposes certain
restrictions on the scope of $\dot u$ that can be considered. With the entropic
criterion, however, it is no longer necessary to solve an SDE; instead, one can pass to the limit in equation
\eqref{eq_hawkes_brouillon:38}, thereby expanding the space of drifts that can
be considered (see also \cite{hartmann2016variational}).

In this work, we follow a similar methodology to establish an analogous formula
for the Poisson space (see
Theorem~\ref{thm_hawkes_core:variational_representation}), based on the Poisson
entropic criterion we establish
(Theorem~\ref{thm_hawkes_core:left_invertible_entropy}): $\Ymap$ is left
invertible if and only if
\begin{equation}
  \label{eq_hawkes_core:37}
  H(\Ymap^{\#}\pi\,|\,\pi)=\espi{\pi}{\int_{0}^{\infty} \Bigl( \dot y^{*}(s)\log \dot y^{*}(s)-\dot y^{*}(s)+1 \Bigr)\dif s},
\end{equation}
where $\dot y^{*}$ is the derivative of the reciprocal function of~$y$.
A similar result has been previously established in
\cite{Zhang2009}, but it requires to  embed the Poisson process within the
broader probability space of marked point processes.


This paper is organized as follows: In Section~\ref{sec:changes-time-changes},
we define what we consider here as random time changes and describe their action
on stochastic integrals. We establish the quasi-invariance theorem and compute
the Radon-Nikodym density of the push-forward of the Poisson measure (see below
for the definition) by a random time change. In Section~\ref{sec:invertibility},
we combine these results to obtain the entropic criterion which guarantees the
invertibility of a map like $\Ymap$. We take profit of the robustness of the
entropic criterion with respect to limit procedure to obtain a variational
representation of the entropy in Section~\ref{sec:vari-repr-entr}. The notion of
weak, strong and martingale g-Hawkes problem are introduced in
Section~\ref{sec:weak-strong-hawkes}. We then prove the Yamada-Watanabe theorem
for the g-Hawkes problem.
\section{Changes of time and changes of measures}
\label{sec:changes-time-changes}
\subsection{Preliminaries}
\label{sec:prelim}
\begin{definition}\label{def_hawkes_core:definition_point_processes}
  Let $\NN$ be the set of locally finite, simple configurations on
  $E=(0,+\infty)$ equipped with the vague convergence. We denote by $\omega$ its
  generic element. For each $\omega\in \NN$, there exist
  $(T_{n}(\omega),n\ge 1)$ such that $ T_{n}(\omega)<T_{n+1}(\omega)$,
  $T_{n}(\omega)$ tends to infinity as $n$ tends to infinity and
  \begin{equation*}
    \omega=\sum_{n\ge 1}\epsilon_{T_{n}(\omega)}
  \end{equation*}
  where $\epsilon_{a}$ is the Dirac mass at~$a$. We denote by $N$, the counting
  process associated to this measure by:
  \begin{equation*}
    N(\omega,t)=\omega\bigl([0,t]\bigr).
  \end{equation*}
  Note that given the sample path of $N$, we can retrieve $\omega$ as
  \begin{equation*}
    T_{n}(\omega)=\inf\{t,\, N(t)=n\} \text{ or } (T_{n}(\omega),\, n\ge 1)=(t, \Delta N(t)=1)
  \end{equation*}
  where $\Delta N(t)=N(t)-N(t^{-})$. This means that we can identify a sample
  path of $N$ with an element of $\NN$. A point process is a random variable
  with values in $\NN$.
\end{definition}
The filtrations do play an important role in the following. We denote by
$\mathcal{F}=(\mathcal{F}_{t},\, t\ge 0)$ a generic right continuous filtration
on $\NN$. We denote by $\F_{\infty}$ the whole $\sigma$-field, i.e.
$\F_\infty=\vee_{t\ge 0}\F_{t}$. The minimal filtration under which the
canonical process $N$ is measurable is
\begin{equation}
  \label{eq_hawkes_core:24}
  \mathcal{N}_{t}=\sigma\{N(s),\ s\le t\}.
\end{equation}
We follow the presentation of \cite{Jacod1975} where the notion of
predictability is defined without any reference to the completion of the
filtration.
\begin{definition}
  For any filtration ${\cal F}=(\mathcal F_{t},\ t\ge 0)$ on $\NN$, a
  real-valued process $(X(t),\,t\ge 0)$ is called $\mathcal F$-predictable and
  belongs to $\pred(\F)$, if the application $(\omega,t)\mapsto X(\omega,t)$ is
  measurable with respect to the $\sigma$-algebra $\pred$ on
  $\NN\times[0,+\infty)$ generated by the applications
  $(\omega,t)\mapsto Y(\omega,t)$ which are $\mathcal F_t$ measurable in
  $\omega$ and left-continuous in $t$. We denote by $\pred(\F)$ the set of
  processes which are $\F$-predictable.
\end{definition}


\begin{theorem}\label{thm_hawkes_core:watanabe_charac} Let $\mathcal{F}$ be a
  filtration on $\NN$ and $\mu$ a probability measure on $(\NN,\F_{\infty})$.
  The next result comes from proposition (3.40) and theorem (3.42)
  of~\cite{JacodCalculstochastiqueproblemes1979}.

  \begin{enumerate}
    \item \label{item_hawkes_core:3} Then, there exists a unique predictable (up
          to $\mu$-null set) process $y$ which is non-decreasing, right
          continuous, null at time $0$ and such that for any $q\ge 1$,
          \begin{equation*}
            t\longmapsto N\bigl(t\wedge T_{q}(N)\bigr) - y\bigl( t\wedge T_{q}(N)\bigr)
          \end{equation*}
          is a uniformly integrable $(\NN,\mathcal{F},\mu)$-martingale.

    \item \label{item_hawkes_core:4} For the converse, we need to assume that
          the filtration is the minimum filtration under which the canonical
          process is measurable. For any process $y\in \pred(\mathcal{N})$,
          non-decreasing, null at time $0$ and right-continuous, there exists a
          unique probability measure on $(\NN,{\mathcal N}_{\infty})$, denoted
          by $\pi_{y}$, such that $N-y$ is a $(\mathcal{N}_{t},\, t\ge 0)$ local
          martingale and $(T_{q}(N),\, q\ge 1)$ is a localizing sequence.

          When $y=\Id$, i.e., when the canonical process is a unit rate Poisson
          process, we prefer to use the notation $\pi$ instead of~$\pi_{\Id}$.
  \end{enumerate}
\end{theorem}

\subsection{Absolute continuity and equivalence}
\label{sec:absol-cont-equiv}
We summarize the results on local absolute continuity on the Poisson space which
can be found in theorems 8.32 and 8.35 and corollary 8.37 of
\cite{JacodCalculstochastiqueproblemes1979}.
\begin{theorem}\label{thm_hawkes_brouillon:absolueContinuite}
  Consider a filtration $\F=(\F_{t},\,t\ge 0)$ on $\NN$. Let $\dot \kappa$ be a
  non-negative, locally integrable $\F$-predictable process. We set
  \begin{equation*}
    \kappa(t)=\int_{0}^{t}\dot \kappa(s)\dif s.
  \end{equation*}
  We consider $\pi_{\kappa}$ as defined in
  Theorem~\ref{thm_hawkes_core:watanabe_charac}.
  \begin{enumerate}
    \item If a probability measure $\nu$ on $\NN$ is locally absolutely
          continuous with respect to~$\pi_\kappa$ along $\F$ (that is denoted
          $\nu \ll_{\text{loc},\F} \mu$) then 
          there exists a unique process $\dot y_\kappa$ which is non-negative
          and $\mathcal{F}$-predictable such that
          \begin{equation}
            \label{eq_hawkes_brouillon:24}      \forall t\ge 0,  \   \nu\left( \int_{0}^{t} \left( 1-\sqrt{\dot y_\kappa(s)} \right)^{2}\dot\kappa(s)\dif s <\infty \right)=1
          \end{equation}
          and
          \begin{equation}
            \label{eq_hawkes_brouillon:25}      t\longmapsto N\bigl( t
            \bigr)-\int_{0}^{t}\dot y_\kappa(s)\dot \kappa(s)\dif s \text{    is a $(\F,\nu)$-local martingale}.
          \end{equation}
  \end{enumerate}
  We set
  \begin{equation*}
    y_\kappa(t)=\int_{0}^{t}\dot y_\kappa(s)\,\dot \kappa(s)\dif s
  \end{equation*}
  which belongs to $\pred(\mathcal{F})$, is non decreasing, right continuous and
  null at time $0$. Hence, with the notations introduced above,
  $\nu=\pi_{y_{\kappa}}$.
  \begin{enumerate}[resume]
    \item If $\pi_{y_\kappa}$ is absolutely continuous with respect to
          $\pi_{\kappa}$ on $\F_{\infty}$ (that is denoted by
          $\pi_{y_\kappa}\ll \pi_{\kappa}$) then
          \begin{equation}\label{eq_hawkes_brouillon:26}
            \pi_{y_\kappa}\left(  \int_{0}^{\infty}\left( 1-\sqrt{\dot y_\kappa(s)} \right)^{2}\dot\kappa(s)\dif s<\infty\right)=1.
          \end{equation}

	\end{enumerate}
	For the converse part, we need to assume that $\F=\mathcal{N}$, the minimal
  filtration which makes $N$ adapted. In this situation, consider $\dot y$ a
  locally integrable, non negative and $\mathcal{N}$-predictable process and
  $\pi_{y_{\kappa}}$ the probability measure on $\NN$, which satisfies
  \eqref{eq_hawkes_brouillon:24} and \eqref{eq_hawkes_brouillon:25}.

	\begin{enumerate}[resume]
	  \item Then, $\pi_{y_{\kappa}}$ is locally absolutely continuous with respect
          to $\pi_{\kappa}$ along $\mathcal{N}$ and
          \begin{multline*}
            \Lambda_{ y_\kappa}(N,t):= \left.\frac{d\pi_{ y_\kappa}}{d\pi_{\kappa}}\right|_{\mathcal{N}_{t}}\\
            =
            \begin{cases}
              \exp\left(\displaystyle\int_{0}^{t} \log \dot y_\kappa(s) \dif N(s)+\displaystyle\int_{0}^{t}\bigl(1-\dot
              y_\kappa(s)\bigr)\dot \kappa(s)\dif s\right) & \text{ if } t\le S_{m}  ,           \\
              0                                            & \text{ if } t\ge \limsup_{m} S_{m},
            \end{cases}
          \end{multline*}
          where for any integer $m\ge 1$,
          \begin{equation*}
            S_{m}=\inf\left\{ t,\, \int_{0}^{t} \left( 1-\sqrt{\dot y_\kappa(s)} \right)^{2}\dot\kappa(s)\dif s\ge m \right\}.
          \end{equation*}
	  \item Finally, the probability measure $\pi_{y_{\kappa}}$ is absolutely
          continuous with respect to~$\pi_{\kappa}$ on
          $(\NN,\mathcal{N}_{\infty})$ if and only if
          \eqref{eq_hawkes_brouillon:25} and \eqref{eq_hawkes_brouillon:26} are
          satisfied.
	\end{enumerate}

\end{theorem}
The next result gives some necessary and sufficient conditions for the
equivalence of $\pi_{y}$ and $\pi_{\kappa}$, see Proposition (7.11) of
\cite{JacodCalculstochastiqueproblemes1979}.
\begin{lemma}
	\label{thm_hawkes_core:Equivalence_criterion}
	Assume that $\nu$ is absolutely continuous with respect to $\mu$ on
  $\mathcal{N}_{\infty}$ and set
	\begin{equation}
	  \label{eq_hawkes_core:8}
	  \Lambda(t)=\left.\frac{\D \nu}{\D \mu}\right|_{\mathcal{N}_{t}}.
	\end{equation}
	Then $\nu$ and $\mu$ are equivalent if and only the following two conditions
  are satisfied:
	\begin{enumerate}[label=\roman*)]
	  \item \label{item_hawkes_core:1} The local martingale
          $(\Lambda(t),\, t\ge 0)$ is uniformly integrable, i.e. there exists
          $\Lambda\in L^{1}(\mu)$ such that
          \begin{equation*}
            \Lambda(t)=\espi{\mu}{\Lambda \/ \mathcal{N}_{t}}.
          \end{equation*}
	  \item \label{item_hawkes_core:2} The random variable $\Lambda $ is positive
          $\mu$-a.s. 
	\end{enumerate}
\end{lemma}
In view of this theorem, we introduce the following sets of processes.
\begin{definition}
	Let $\F$ be a filtration on $\NN\times \R^{+}$ and $\dot \kappa$ a
  non-negative predictable process. Consider the probability measure
  $\pi_{\kappa}$ on $\NN$ and let $\pred^{++}(\F,\pi_{\kappa})$ (respectively
  $\pred^{+}(\F)$) be the set of positive (respectively non-negative)
  $\F$-predictable processes $\dot y $ such that for any $t\ge 0$
	\begin{equation*}
	  \pi_\kappa\Bigl( y_{\kappa}(t)<+\infty \Bigr)=1 \text{ and }    \pi_\kappa\left(     \lim_{t\to \infty} y_{\kappa}(t)=+\infty \right)=1,
	\end{equation*}
	where the process $y$ is defined by
	\begin{equation*}
	  y_{\kappa}(\omega,t)=\int_{0}^{t}\dot y(\omega,s)\dot \kappa(s)\dif s.
	\end{equation*}
	We also introduce the following subset of $\pred^{++}(\F,\pi_\kappa)$:
	\begin{align*}
	  \pred^{++}_{2}(\F,\pi_\kappa) & =\left\{ \dot y \in \pred^{++}(\F,\pi_\kappa), \ \pi_\kappa\left( \int_{0}^{\infty}\left( 1-\sqrt{\dot y(s)} \right)^{2}\dot \kappa(s)\dif s <\infty\right)=1 \right\}.
	\end{align*}
	The most restricted class we consider is
  $ \pred^{++}_{\infty}(\F,\pi_{\kappa})$ of processes
  $\dot y\in \pred^{++}(\F,\pi_\kappa)$ for which there exist
  $\varepsilon\in (0,1)$ and $T>0$ such that
	\begin{equation*}
	  \varepsilon\le \dot y(s)\le \frac{1}{\varepsilon},\ \forall s\ge 0\text{ and }  \dot y(s)=1\text{ for  }s\ge T, \ \pi_\kappa-\text{a.s.}
	\end{equation*}
	Note that if $\dot y $ belongs to $ \pred^{++}_{\infty}(\F,\pi_{\kappa})$ then
  $(\Lambda_{y_{\kappa}}(N,t),\, t\ge 0)$ is uniformly integrable.
\end{definition}

As long as there is a finite number of jumps of the process~$N$, which happens
on any finite interval, the argument of the exponential in
$\Lambda_{y_{\kappa}}$ cannot be minus infinity hence $\Lambda_{y_{\kappa}}$ is
positive. We give a sufficient condition which ensures that this may not happen
even on the half line.
\begin{lemma}\label{lem_hawkes_core:Conditionpositivity}
	If the predictable process $(\dot y-1)$ belongs to
  $L^{1}\bigl(\NN\times \R^{+},\, \pi_\kappa\otimes \dot \kappa(s)\dif s\bigr)$
  then, $\pi_\kappa$-a.s.,
	\begin{align}
	  \Lambda_{y_\kappa} & :=\lim_{t\to \infty} \Lambda_{y_{{\kappa}}}(t)\notag                                                                                                              \\
                       & = \exp\left( \int_{0}^{\infty}\log\dot y(s)\dif N(s)+\int_{0}^{\infty}\bigl(1-\dot y(s)\bigr) \dot \kappa(s)\dif s\right)\in (0,+\infty).\label{eq_hawkes_core:2}
	\end{align}
\end{lemma}
\begin{proof}

	Since, for any $x\ge -1$,
	\begin{equation*}
	  \left(1-\sqrt{1+x}\right)^{2}\le |x|,
	\end{equation*}
	we have
	\begin{equation*}
	  \espi{\pi_\kappa}{\int_{0}^{\infty} \left( 1-\sqrt{\dot y(s)} \right)^{2}\dot \kappa(s)\dif s}\le \espi{\pi_\kappa}{\int_{0}^{\infty} \left|\dot y(s)-1\right|\ \dot \kappa(s)\dif s}<\infty.
	\end{equation*}
	Hence, for $\pi_\kappa$ almost all $\omega$, there exists $m(\omega)<\infty$
  such that $S_{m}(\omega)=\infty$ for any $m\ge m(\omega)$ ($S_m$ is defined in
  Theorem \ref{thm_hawkes_brouillon:absolueContinuite} point 3) and then
	we have
	\begin{align*}
	  \Lambda_{y_{{\kappa}}} & =\lim_{t\to \infty}\Lambda_{y_{{\kappa}}}(t)                                                                                                      \\
                           & = \left(\prod_{n=1}^{\infty}\dot y\left( T_{n} \right)\right)\ \exp\left( \int_{0}^{\infty}\left(1-\dot y(s)\right) \dot \kappa(s)\dif s \right).
	\end{align*}
	The infinite product is convergent as soon as
	\begin{equation*}
	  \sum_{n=1}^{\infty} \bigl|\dot y(T_{n})-1\bigr|<\infty.
	\end{equation*}
	Since
	\begin{equation*}
	  \espi{\pi_\kappa}{\sum_{n=1}^{\infty} \left|\dot y(T_{n})-1\right|}=\espi{\pi_\kappa}{\int_{0}^{\infty}\left| \dot y(s)-1 \right|\ \dot \kappa(s)\dif s},
	\end{equation*}
	we see that
  $(\dot y-1)\in L^{1}\bigl(\NN\times \R^{+},\, \pi_\kappa\otimes \dot \kappa(s)\dif s\bigr)$
  entails \eqref{eq_hawkes_core:2}.

\end{proof}

\subsection{Random time changes}
\label{sec:random-time-changes}

The sequel of this paper will be based on the notion of random time change. We
refer to \cite[chapter 10]{JacodCalculstochastiqueproblemes1979} or
\cite{MR3363697} for details on this notion and its links with stochastic
calculus.
\begin{definition}
	On $\bigl(\NN,\F=(\F_{t},\, t\ge 0)\bigr)$, a random change of time is a right
  continuous, null at time $0$, non-decreasing process $(\eta(t),\, t\ge 0)$
  such that
	\begin{equation*}
	  \eta(t)<\infty, \ \forall t\ge 0 ,\  \lim_{t\to  \infty} \eta(t)=+\infty
	\end{equation*}
	and for any $t\ge 0$, $\eta(t)$ is an $\F$-stopping time. We denote by
  $\F^{\eta}$ the filtration $(\F_{\eta(t)},\, t\ge 0)$. For $X$ a process,
  $\tau_{\eta} (X)$ is the process defined by
	\begin{equation*}
	  \tau_{\eta} (X)(t)=
	  X\Bigl(\eta(t)\Bigr).
	\end{equation*}
\end{definition}
Note that many changes of time are given by the right inverse of non decreasing
predictable processes: For $y$ such a process,
\begin{equation*}
	y^{*}(\omega,t)=\inf\left\{s,\, y(\omega,s)>  t\right\}
\end{equation*}
is a change of time (with the usual convention that the infimum of the empty set
is infinite). According to \cite[Chapter
10]{JacodCalculstochastiqueproblemes1979},
\begin{align}
	\label{eq_hawkes_core:25}
	\bigl(y(\omega,t)<s\bigr)        & =\bigl(y^{*}(\omega,\, s^{-})>t\bigr)                      \\
	\bigl(y ^{*}(\omega,\,t)<s\bigr) & =\bigl(y(\omega,\ s^{-})>t\bigr)\label{eq_hawkes_core:28}.
\end{align}
If $y$ is continuous, we have
\begin{equation*}
	y\bigl(\omega,\, y^{*}(\omega,t)\bigr)=t \text{ and
	} y^{*}\bigl(\omega,\ y(\omega,t)\bigr)=\inf\{u,\ y^{*}(\omega,u)>y^{*}(\omega,t)\}.
\end{equation*}
Note that if $\dot y$ belongs to $\pred^{++}(\F,{\mu})$, then $y$ is
almost-surely an homeomorphism from $\R^{+}$ onto itself, hence
\begin{equation}\label{eq_hawkes_brouillon:3}
	y\bigl(\omega,y^{*}(\omega,t)\bigr)=y^{*}\bigl(\omega,y(\omega,t)\bigr)=t \text{
	  and } \tau_{y^{*}}(X)(t)=X\Bigl( y^{*}(t) \Bigr), \text{ for all }t\ge 0,\ \mu\text{-a.s.}
\end{equation}
Furthermore, \eqref{eq_hawkes_core:25} entails that $y(\omega,t)$ is an
$\F^{y^{*}}$ stopping time.

In the Brownian setting, the entropic criterion involves the $L^{2}$ norm of the
drift. The square function is here replaced by a convex function which appears
frequently in Poissonian settings (see for instance \cite{Atar2012} and
references therein).
\begin{definition}
	Consider $\m$, the smooth, convex, non-negative function defined on
  $[-1,\infty)$ by
	\begin{equation*}
	  \m(x)=
	  \begin{cases}
      (x+1)\log(x+1)-x & \text{ if } x >-1, \\
      1                & \text{ if } x=-1.
	  \end{cases}
	\end{equation*}
	and $\L_{\m}$, the corresponding Orlicz space \cite{AdamsSobolevspaces2003}:
	\begin{equation*}
	  \L_{\m}=\Bigl\{f\, :\,\R^{+}\to [-1,\infty),\ \int_{0}^{\infty} \m\bigl(f(s)\bigr)\dif s<\infty\Bigr\}.
	\end{equation*}
\end{definition}
We view the time change as a map from $\NN$ into itself.
\begin{definition}\label{def_hawkes_brouillon:1}
	For $\dot y\in \pred^{++}(\mathcal{N})$, let
	\begin{align}
	  \tau_{y^{*}} :\, \NN & \longrightarrow \NN\notag                                                      \\\label{eq_hawkes_brouillon:2}
	  N                    & \longmapsto \tau_{y^{*}}(N)\,:\, \Bigl(t\mapsto N\bigl(y^{*}(N,t)\bigr)\Bigr).
	\end{align}
	We denote by $Y$ the process $\tau_{y^{*}}(N)$.
\end{definition}
Note that the process $Y$ is adapted to the filtration
\begin{equation*}
	\F^{y^{*}}=\left(\F_{y^{*}(t)},\, t\ge 0\right)
\end{equation*}
and has $(\F^{y^{*}},\pi)$-compensator~$y^{*}$ (see \cite[Theorem
10.17]{JacodCalculstochastiqueproblemes1979}).
\begin{definition}
	For a probability measure $\mu$ on $\NN$, we denote by $\tau_{y^{*}}^{\#}\mu$
  the distribution of the process~$Y$ on~$\NN$ or equivalently the push-forward
  of the measure $\mu$ by the map~$\tau_{y^{*}}$.

\end{definition}
One way to understand this transformation is to consider that the locations of
the atoms of the underlying configuration $N$ are fixed once and for all. They
are discovered at unit rate with~$N$ and at random speed by~$Y$: The $k$-th jump
of $Y $ occurs at time $t$ if and only $y^{*}({N},t)=T_{k}(N)$ hence
\begin{equation}
  \label{eq_hawkes_brouillon:sautsDeY}
  y^{*}(N,\, T_{k}(Y))=T_{k}(N).
\end{equation}

\begin{lemma}\label{lem:EgaliteFiltration}
  For $\dot y\in \pred^{++}(\mathcal{N})$, let
  $\mathcal{Y}_{t}=\sigma\bigl(Y(u),\, u\le t\bigr)$. Then, we have
  \begin{equation*}
    \mathcal{Y}_{t}\vee \sigma(y^{*}(N,s),\, s\le t)=\mathcal{N}_{y^{*}(N,t)}.
  \end{equation*}
  Furthermore,
  $\mathcal{Y}_{\infty}\subset \mathcal{N}^{y^{*}}_{\infty}=\mathcal{N}_{\infty}$.
\end{lemma}
\begin{proof}
  For any $s\le t$,
  \begin{equation*}
    y^{*}(N,s)\in \mathcal{N}_{y^{*}(N,s)}\subset \mathcal{N}_{y^{*}(N,t)}.
  \end{equation*}
  The definition of $Y$ induces that $Y(t)$ is $\mathcal{N}_{y^{*}(t)}$
  measurable hence
  \begin{equation*}
    \mathcal{Y}_{t}\subset\mathcal{N}_{y^{*}(N,t)}.
  \end{equation*}
  It follows that
  \begin{equation*}
    \mathcal{Y}_{t}\vee \sigma(y^{*}(N,s),\, s\le t)\subset\mathcal{N}_{y^{*}(N,t)}.
  \end{equation*}
  Conversely, let $A\in \mathcal{N}_{y^{*}(N,t)}$, according to
  \cite[Proposition 3.40]{JacodCalculstochastiqueproblemes1979},
  \begin{equation*}
    \car_{A}= \sum_{q=0}^{\infty} \Psi_{q}\Bigl(T_{1}(N),\cdots,T_{q}(N)\Bigr)\ \car_{\{T_{q}(N)\le y^{*}(N,t)<T_{q+1}(N)\}},
  \end{equation*}
  where $\Psi_{q}$ is measurable from $(\R^{+})^{q}$ to $\{0,1\}$. It follows
  from the definition of $Y$ (see \eqref{eq_hawkes_brouillon:sautsDeY}) that
  \begin{equation*}
    \car_{A}= \sum_{q=0}^{\infty} \Psi_{q}\Bigl(  y^{*}\bigl(N,\,T_{k}(Y)\bigr),\, k=1,\cdots,q \Bigr)\ \car_{\{Y(t)=q\}}.
	\end{equation*}
	{Since $y^*$ is continuous, for $u \leq t$, we have
	\begin{align*}
		y^*(N,u)=\lim_{n \rightarrow \infty} \sum_{i=0}^{2^n-1} y^*(N,\frac{i}{2^n}){\mathbf 1}_{[\frac{(i-1)}{2^n}, \frac{i}{2^n})}(u),
	\end{align*}
	thus, for $k \leq q$,
	\begin{align*}
		y^*(N,T_k(Y)) \car_{\{Y(t)=q\}}\in {\mathcal Y}_t \wedge \sigma(y^*(s),~~s\leq t).
	\end{align*}}
	Hence the converse embedding also holds.

	The inclusion $\mathcal{Y}_{\infty}\subset \mathcal{N}^{y^{*}}_{\infty}$ is
  immediate from the first part of the proof. The equality between
  $ \mathcal{N}^{y^{*}}_{\infty} $ and $ \mathcal{N}_{\infty}$ comes from
  \cite[page 326]{JacodCalculstochastiqueproblemes1979}.
\end{proof}
The following theorem is Theorem (10.19) of
\cite{JacodCalculstochastiqueproblemes1979} and Lemma 1.3 of \cite{MR3363697}.
\begin{theorem}\label{thm_hawkes_brouillon:chgtdetemps}
	Let $\dot y \in \pred^{++}(\F)$. For $r\in \pred(\mathcal{F})$ such that its
  stochastic integral is well defined, we have
	\begin{align}
	  \tau_{y^{*}} \left( \int_{0}^{.} r(N,\,s)\dif s \right)  & =\int_{0}^{.} r\bigl(N,\,y^{*}(N,\,s)\bigr) \dif y^{*}(s)\label{eq_hawkes_brouillon:6}
	  \\
	  \intertext{and for the stochastic integrals,}
	  \tau_{y^{*}}\left( \int_{0}^{.} r(N,\,s)\dif N(s)\right) & =\int_{0}^{.} r(N,\,y^{*}(N,\,s))\dif Y(s).
                                                               \label{eq:chgttempsIntSto}
	\end{align}
\end{theorem}
We can extend $\tau_{y^{*}}$ to stochastic processes which are measurable maps
from $(\NN\times \R^{+},\, {\cal N}_{\infty}\otimes {\cal B}(\R^{+}))$ to
$(\R,\, {\cal B}(\R))$. For $\mu$ a probability measure on $\NN,$ we denote by
$L^{0}(\NN\times \R^{+},\, \R;\mu)$ the set of stochastic processes equipped
with the topology of convergence in probability. We denote by $\Ymap$ this
extension, which is not to be confused with the process~$Y$. For the sake of
simplicity, we henceforth use the notation $\Ymap$ even for $\tau_{y^{*}}$.
\begin{definition}
  \label{def_hawkes_core:1}
	The map $\Ymap$ is defined by
	\begin{align*}
	  \Ymap\, :\, L^0\bigl(\NN\times \R^{+},\, \R;\mu\bigr) & \longrightarrow   L^0\bigl(\NN\times \R^{+},\, \R;\mu\bigr)           \\
	  U                                                     & \longmapsto\Bigl( (N,\, s)\mapsto U\bigl(Y,\, y^{*}(N,s)\bigr)\Bigr).
	\end{align*}
\end{definition}
For instance, if we denote by $N^{a}$ the process $N$ stopped at time~$a$:
\begin{equation*}
	N^{a}(t)=N\bigl(t\wedge a\bigr),
\end{equation*}
we have
\begin{align*}
	Y^{y(a)}(t) & =Y\left(y(N,a)\wedge t\right)             \\
	            & =N\left(y^{*}(N,\, y(N,a)\wedge t)\right) \\
	            & =N\left(a\wedge y^{*}(N,t)\right).
\end{align*}
On the other hand, we have
\begin{align*}
	y^{*}(N^{a})(t) & =N^{a}\left(y^{*}(N,t)\right)    \\
	                & =N\bigl(y^{*}(N,t)\wedge a\bigr)
\end{align*}
so that
\begin{equation}
	\label{eq:chgt}
	N^{a}\circ \Ymap=Y^{y(N,a)}.
\end{equation}
We then have the following composition rules.
\begin{theorem}

	\label{thm_hawkes_core:compositionTimeChange}
	Let $\dot y \in \pred^{++}(\mathcal{N})$. For $r\in \pred(\mathcal{N})$ such
  that its stochastic integral with respect to $Y$ is well defined, we have
	\begin{align}
	  \left( \int_{0}^{.} r(N,\,s)\dif s \right)\circ \Ymap   & =\int_{0}^{.} r\bigl(Y,\, y^{*}(N,\,s)\bigr) \dif y^{*}(N,\,s)\label{eq_hawkes_brouillon:6b}
	  \\
	  \intertext{and}
	  \left( \int_{0}^{.} r(N,\,s)\dif N(s)\right)\circ \Ymap & =\int_{0}^{.} r(Y,\, y^{*}(N,\,s))\dif Y(s).
                                                              \label{eq:chgttempsIntStob}
	\end{align}
\end{theorem}
\begin{remark}
	These formulas are no longer valid if $\dot y$ is not supposed to be positive
  or if $y$ is not an homeomorphism almost surely.
\end{remark}
\begin{proof}
	Since $y$ defines a diffeomorphism from $\R^{+}$ onto itself, the change of
  variable $u=y(N,s)$ yields
	\begin{align*}
	  \left( \int_{0}^{.} r(N,s)\dif s \right)\circ \Ymap(t) & =\int_{0}^{y^{*}(N,t)} r(Y,\,s)\dif s                              \\
                                                           & =\int_{0}^{t} r\bigl(Y,\,y^{*}(N,\,u)\bigr)\dot y^{*}(N,\,u)\dif u
	\end{align*}
	and \eqref{eq_hawkes_brouillon:6b} holds.

	To prove \eqref{eq:chgttempsIntStob}, it is sufficient to prove it for simple
  predictable process, i.e. we assume that
	\begin{equation*}
	  r(N,s)=A(N)\ \car_{(a,b]}(s)
	\end{equation*}
	for some $A\in \mathcal{N}_{a}$. We have
	\begin{equation*}
	  \int_{0}^{.}r(N,s)\dif N(s)=A(N)\bigl( N^{b}-N^{a} \bigr)
	\end{equation*}
	where $N^{a}$ is the process $N$ stopped at time $a$. On the one hand,
  according to the definition of $\Ymap$ and to \eqref{eq:chgt}, we have
	\begin{equation*}
	  \Bigl(    A(N)\bigl( N^{b}-N^{a} \bigr)\Bigr)\circ \Ymap
	  =A(Y)\Bigl( Y^{y(b)}-Y^{y(a)} \Bigr).
	\end{equation*}
	On the other hand,
	\begin{equation*}
	  r(Y,y^{*}(N,s))=A(Y)\ \car_{(y(N,a),y(N,b)]}(s)
	\end{equation*}
	hence
	\begin{equation*}
	  \int_{0}^{.} r\bigl(Y,\,y^{*}(N,s)\bigr)\dif Y(s)=A(Y)\ \Bigl( Y^{y(N,b)}-Y^{y(N,a)} \Bigr).
	\end{equation*}
	The proof is thus complete.
\end{proof}

\subsection{Quasi-invariance}
\label{sec:girsanov-theorem}

The next result is the exact analog of the quasi-invariance theorem for the
Wiener space (often quoted as the Girsanov theorem): find a perturbation of the
sample paths and a change of probability which compensate each other. As
mentioned above, for point processes, translations by an element of the
Cameron-Martin space are replaced by time changes (see
\cite{DecreusefondPerturbationAnalysisMalliavin1998} and
\cite{BassMalliavinCalculusPure1986} for marked point processes). We first
evaluate how a time change modifies the Radon-Nikodym derivative.

\begin{lemma}\label{lem_hawkes_core:1}
  Let $\dot \kappa$ be a deterministic positive function from $\R^{+}$ into
  itself such that
  \begin{equation*}
    \kappa(t)=\int_{0}^{t}\dot \kappa(s)\dif s<\infty, \ \forall t\ge 0 \text{ and } \lim_{t\to \infty}\kappa(t)=+\infty
    .
  \end{equation*}
  Let $\nu$ be locally absolutely continuous with respect to $\pi_{\kappa}$
  along $\F$ and let $\dot y_{\kappa}$ denote its Girsanov factor, i.e.
  $\nu=\pi_{y_{\kappa}}$. Assume that $\dot y_{\kappa}$ belongs to
  $\pred^{++}_{2}(\F,\pi_{y_{\kappa}})$. 
  Then, $\pi_{y_{\kappa}}$ is locally absolutely continuous with respect to
  $\pi_{\kappa}$ along $\F^{y_{\kappa}^{*}}$ and we have
  \begin{multline}
    \Lambda_{y_{\kappa}^{*}}^{*}(t) :=               \left.\frac{d\nu}{d\pi{}}\right|_{\F^{y_{\kappa}^*}_{t}}
    =  \exp\left( \int_{0}^{t}\log\left( \dfrac{1}{(\kappa\circ y_{\kappa}^{*})'(s)}
      \right)\dif Y(s)+\int_{0}^{ t}
      \Bigl((\kappa\circ y_{\kappa}^{*})'(s)-1 \Bigr)\dif s\right).
    \label{eq_hawkes_brouillon:9}
  \end{multline}
\end{lemma}

\begin{proof}
  Remark that under the hypothesis
  $\dot y_{\kappa} \in \pred^{++}_{2}(\F,\pi_{y_{\kappa}})$, for almost all
  sample paths, $S_{m}$ is infinite after a certain rank thus for any $t\ge 0$,
  \begin{equation}\label{eq_hawkes_core:17}
    \Lambda_{y_{{\kappa}}}(t)= \exp\left(\displaystyle\int_{0}^{t} \log \dot y(s) \dif N(s)+\displaystyle\int_{0}^{t}\bigl(1-\dot
      y(s)\bigr)\dot\kappa(s)\dif s\right).
  \end{equation}
  For $A\in\F_{y_{\kappa}^{*}(t)}$, by monotone convergence, we have
  \begin{align*}
    {\mathbf E}_{\pi_{y_{\kappa}}}[\car_{A}] & =\lim_{s\to \infty}\espi{\pi_{y_{{\kappa}}}}{\car_{A}\car_{\{y_{\kappa}^{*}(t)\le s\}}}                                                                                                                                                \\
                                             & =\lim_{s\to \infty}\lim_{n\to \infty}\espi{\pi_{{\kappa}}}{\car_{A}\car_{\{s <S_n\}}\car_{\{y_{\kappa}^{*}(t)\le s\}}\ \Lambda_{y_{{\kappa}}}}                                                                                         \\
                                             & =\lim_{s\to \infty}{\lim_{n\to \infty}}\  \espi{\pi_{{\kappa}}}{\car_{A}\car_{\{s <S_n\}}\car_{\{y_{\kappa}^{*}(t)\le s\}}\ \espi{\pi_{{\kappa}}}{\Lambda_{y_{\kappa}}(s)\, \left|\, \F_{y_{\kappa}^{*}(t)\wedge s\wedge S_n}\right.}} \\
                                             & =\lim_{s\to \infty}\lim_{n\to \infty}\espi{\pi_{{\kappa}}}{\car_{A}\car_{\{s<S_n\}}\car_{\{y_{\kappa}^{*}(t)\le s\}}\ {\Lambda_{y_{\kappa}}\Bigl(y_{\kappa}^{*}(t)\wedge s\wedge S_{n}\Bigr)}}\
  \end{align*}
  according to the stopping time theorem applied to the bounded stopping time
  $y_{\kappa}^{*}(t)\wedge s$ and to the martingale
  $\Lambda_{y_{\kappa}}^{S_{n}}$. 
  Thus, in view of \eqref{eq_hawkes_core:17}, we obtain
  \begin{equation*}
    \espi{\pi_y}{\car_{A}}=\espi{\pi}{\car_{A}\ {\Lambda_{y}\Bigl(y_{\kappa}^{*}(t)\Bigr)}}.
  \end{equation*}
  Then, Theorem~\ref{thm_hawkes_brouillon:chgtdetemps} yields
  \begin{align*}
    \Lambda_{y_{{\kappa}}}\Bigl(y_{\kappa}^{*}(t)\Bigr) & =\exp\left( \int_{0}^{t} \log \dot y\bigl(y_{\kappa}^{*}(s)\bigr)\dif Y(s)+\int_{0}^{t} \left(1-\dot y\bigl(y_{\kappa}^{*}(s)\bigr)\right)\, \dot \kappa\bigl(y_{\kappa}^{*}(s)\bigr)\dot y_{\kappa}^{*}(s)\dif s\right).
  \end{align*}
  On the one hand, we have $ y_{\kappa}'=\dot y \dot \kappa$ and on the other
  hand,
  \begin{equation*}
    y_{\kappa}\bigl( y_{\kappa}^{*}(s) \bigr)=s\Longrightarrow  y_{\kappa}'\bigl( y^{*}_{\kappa}(s)\bigr)\ (y_{\kappa}^{*})'(s)=1.
  \end{equation*}
  It follows that
  \begin{equation*}
    \dot y\bigl(y_{\kappa}^{*}(s)\bigr)\  \dot \kappa\bigl(y_{\kappa}^{*}(s)\bigr)\  \bigl(y_{\kappa}^{\kappa}\bigr)'(s)=1
  \end{equation*}
  and that
  \begin{equation*}
    \dot  y\bigl(y_{\kappa}^{*}(s)\bigr)=\frac{1}{(\kappa\circ y_{\kappa}^{*})'(s)}\cdotp
  \end{equation*}
  Thus, \eqref{eq_hawkes_brouillon:9} holds.
\end{proof}
\begin{theorem}[Quasi-invariance]\label{thm_hawkes_brouillon:Pregirsanov}

  Assume that the hypothesis of Lemma~\ref{lem_hawkes_core:1} hold. Then, the
  distribution of the process $Y=\tau_{y_{\kappa}}(N)$ under
  $\pi_{y_{{\kappa}}}$ is $\pi$ . This means that for any bounded measurable
  $f\, :\, (\NN,{\cal N}_{\infty})\to \R$, for any $t\ge 0$,
  \begin{equation}\label{eq_hawkes_core:eq_girsanov}
    \espi{\pi_{\kappa}}{ f(Y^{t})\ \Lambda_{y_{\kappa}^{*}}^{*}(t) }={\mathbf E}_{\pi_{\Id}}[f],
  \end{equation}
  where $Y^{t}$ is the process $Y$ stopped at time~$t$.
\end{theorem}

\begin{proof}
  Note that by the stopping time theorem and Theorem 10.27 of
  \cite{JacodCalculstochastiqueproblemes1979}
  \begin{equation*}
    \espi{\pi_{\kappa}}{Y(t)-\kappa\bigl(y^{*}_{\kappa}(t)  \bigr)\,\left| \, \mathcal{N}_{y_{\kappa}^{*}({s})}\right.}=Y(s)-\kappa\bigl(y^{*}_{\kappa}(s)  \bigr),
  \end{equation*}
  hence the compensator of $Y$ under $\pi_{\kappa}$ is
  $\kappa\circ y^{*}_{\kappa}$. Thus,
  \begin{align*}
    R(t) & =\int_{0}^{t}  \frac{1}{\bigl( \kappa\circ y^{*}_{\kappa} \bigr)'(s)} \dif Y(s)-\int_{0}^{t }\frac{1}{\bigl( \kappa\circ y^{*}_{\kappa} \bigr)'(s)}\ \bigl( \kappa\circ y^{*}_{\kappa} \bigr)'(s)\dif s \\
         & =\int_{0}^{t}  \frac{1}{\bigl( \kappa\circ y^{*}_{\kappa} \bigr)'(s)} \dif Y(s)-t
  \end{align*}
  is a $(\F^{y_{\kappa}^*},\pi_{\kappa})$ local
  martingale. 
  The standard Girsanov theorem \cite[Theorem
  7.24]{JacodCalculstochastiqueproblemes1979} states that
  \begin{equation*}
    R(t)-\int_{0}^{t}\frac{1}{\Lambda_{y_{\kappa}^*}^*(s)}\dif \,[R,\Lambda_{y^*_{{\kappa}}}^*](s)
  \end{equation*}
  is a $(\F^{y_{\kappa}^*},\pi_{y_{\kappa}})$ local martingale. Note that $R$
  and $\Lambda_{y_{\kappa}^*}^*$ have the same jump times as~$Y$, hence
  \begin{align*}
    \int_{0}^{t}\frac{1}{\Lambda_{y_{\kappa}^*}^*(s)}\dif\, [R,\Lambda_{y_{\kappa}^*}^*](s) & =\sum_{s\le t,\Delta Y(s)\neq 0} \frac{1}{\Lambda_{y_{\kappa}^*}^*(s)}\,\Delta R(s)\ \Delta \Lambda_{y_{\kappa}^*}^*(s)                                 \\
                                                                                            & =\sum_{s\le t,\Delta Y(s)\neq 0} \bigl(1-\frac{\Lambda_{y_{\kappa}^*}^*(s^-)}{\Lambda_{y_{\kappa}^*}^*(s)}\bigr)\ \Delta R(s)                           \\
                                                                                            & =\sum_{s\le t,\Delta Y(s)\neq 0} \bigl(1-\bigl( \kappa\circ y^{*}_{\kappa} \bigr)'(s)\bigr)\frac{1}{\bigl( \kappa\circ y^{*}_{\kappa} \bigr)'(s)}\cdotp
  \end{align*}
  Thus, we have
  \begin{align*}
    R(t)- \int_{0}^{t}\frac{1}{\Lambda_{y^*}^*(s)}\dif \,[R,\Lambda_{y^*}^*](s) & =\int_{0}^{t}\left[ \frac{1}{\bigl( \kappa\circ y^{*}_{\kappa} \bigr)'(s)}-\left(\frac{1}{\bigl( \kappa\circ y^{*}_{\kappa} \bigr)'(s)}-1\right)  \right] \dif Y(s)-t \\
                                                                                & =Y(t)-t.
  \end{align*}
  This means that $Y$ has $(\F^{y_{\kappa}^{*}},\pi_{y_{{\kappa}}})$-compensator
  $(t\mapsto t)$. According to Theorem~\ref{thm_hawkes_core:watanabe_charac},
  $Y$ is an $\F^{y_{\kappa}^{*}}$-adapted unit Poisson process
  under~$\pi_{y_{{\kappa}}}$.
\end{proof}

\begin{remark}
  The process $Y_{\kappa}=(Y\bigl( \kappa(t) \bigr),\ { t \geq 0)}$ is then
  adapted to the filtration
  $\mathcal{G}=(\mathcal{N}^{y_{\kappa}^{*}}_{\kappa(t)},\, t\ge 0)$ and is a
  $\pi_{{\kappa}}$ point process of compensator $\kappa$ and then we have
  \begin{equation*}
    \espi{\pi_{\kappa}}{ f(Y^{t}_{\kappa})\ \Lambda_{y_{\kappa}^{*}}^{*}\bigl(\kappa(t)\bigr) }=\espi{\pi_{\kappa}}{f}.
  \end{equation*}
  This invariance formula is the key to our investigations. In full generality,
  we could continue in such a general setting with any $\dot \kappa$. Actually,
  changing $\dot \kappa$ amounts to change the clock with which time is
  measured. As long as this clock is deterministic and increasing, this does not
  change the essence of the results to come so there is no loss in generality
  but a great gain in simplicity and clarity to focus on the situation where
  $\kappa=\Id$. For the sake of notations, we suppress the subscript $\kappa$
  hereafter.
\end{remark}
The following theorems are often written in terms of $\dot y^{*}$, this can be
translated in terms of~$\dot y$ as shows the next lemma.
\begin{lemma}\label{lem-egal-lambda-lambda*}
  For $\dot y \in \pred^{++}_{2}(\mathcal{N})$,
  \begin{align}
    \Lambda_{y} & =\exp\left( \int_{0}^{\infty} \log\bigl( \dot y(N,s) \bigr)\dif N(s)+\int_{0}^{\infty} \left( 1-\dot y(N,s) \right) \dif s\right)\notag \\
                & =\Lambda^{*}_{y^{*}}  \label{eq_hawkes_core:1}                                                                                          \\
                & =\exp\left( -\int_{0}^{\infty}\log\bigl( \dot y^{*}(s) \bigr)\dif Y(s) +\int_{0}^{\infty}\bigl(\dot y^{*}(s)-1\bigr)\dif s \right
                  )\label{eq_hawkes_core:4}
  \end{align}
  Furthermore, for any $p\ge 1$,
  \begin{align}
    \label{eq_hawkes_core:3}
    \int_{0}^{\infty} \left(  \dot y^{*}(s)\log\bigl( \dot y^{*}(s) \bigr)-\dot y^{*}(s)+1 \right) \dif s & =\int_{0}^{\infty}\left(  \dot y(N,s)-1-\log\bigl( \dot y(N,s) \bigr) \right)\dif s \\
    \int_{0}^{\infty} \left| \frac{1}{\dot y^{*}(s)}-1 \right|^{p}\dot y^{*}(s)\dif s                     & =\int_{0}^{\infty}\left| \dot y(N,s)-1 \right|^{p}\dif s.\label{eq_hawkes_core:5}
  \end{align}
\end{lemma}
\begin{proof}
  According to Theorem \ref{thm_hawkes_brouillon:absolueContinuite} point 4 for
  $\dot{y}\in {\mathcal P}_2^{++}({\mathcal N})$ and to
  Lemma~\ref{lem:EgaliteFiltration}, we have
  \begin{equation*}
    \Lambda_{y}=\left.\frac{\D \pi_{y}}{\D \pi}\right|_{\mathcal{N}_{\infty}}=\left.\frac{\D \pi_{y}}{\D \pi}\right|_{\mathcal{N}^{y^{*}}_{\infty}}.
  \end{equation*}
  In view of Theorem 8.1 of \cite{JacodCalculstochastiqueproblemes1979}, this
  means $\pi_{y}$ is absolutely continuous with respect to $\pi$ along the
  filtration $({\cal N}^{y^{*}}_{t},\, t\ge 0)$ and \eqref{eq_hawkes_core:1}
  follows. Eqn.~\eqref{eq_hawkes_core:4} is a straightforward consequence
  of~\eqref{eq_hawkes_brouillon:9}. Identities \eqref{eq_hawkes_core:3} and
  \eqref{eq_hawkes_core:5}, since the integrands are non negative, are obtained
  through the change of variable $s=y(N,u)$ and the relation
  \begin{equation*}
    \dot y^{*}\left(  y(N,u) \right)=\frac{1}{\dot y(N,u)}\cdotp
  \end{equation*}
  The proof is thus complete.
\end{proof}
We now prove that well behaved time changes induce locally absolutely continuous
probability on
$\NN$. Recall that $\Ymap$ is defined in Definition~\ref{def_hawkes_core:1} as
the extension to processes of the time change $y^{*}$. In the following, we
identify the configurations $Y=N\circ y^{*}$ and $\Ymap(N).$ The push forward or image
measure of $\pi$ by $\Ymap$ is denoted by $\Ymap^{\#}\pi.$
\begin{theorem}
  \label{thm_hawkes_brouillon:AbsolueContinuiteChangementTemps}
  Let $\dot y$ belong to $\pred^{++}(\mathcal{N})$ such that $\pi_{y}$ is
  equivalent to $\pi$ on $\mathcal{N}_{\infty}$. Then $\Ymap^{\#}\pi$ is
  equivalent to $\pi$ on $\mathcal{N}_{\infty}$.
\end{theorem}
\begin{proof}
  Since $\pi_{y}$ is equivalent to $\pi$ on $\mathcal{N}_{\infty}$,
  $ \Ymap^{\#}\pi_{y}$ is equivalent to $\Ymap^{\#}\pi$.

  Furthermore, we have
  \begin{equation*}
    \Lambda^{*}_{y^{*}(t)}=\left.\frac{{\D \pi_{y}}}{\D \pi}\right|_{\mathcal{N}_{y^{*}(t)}},
  \end{equation*}
  hence,
  \begin{equation*}
    \lim_{t\to \infty} \Lambda^{*}_{y^{*}(t)}=\left.\frac{{\D \pi_{y}}}{\D \pi}\right|_{\mathcal{N}_{\infty}}=\Lambda_{y}.
  \end{equation*}
  According to Lemma~\ref{thm_hawkes_core:Equivalence_criterion}, the martingale
  \begin{equation*}
    \left(\Lambda^{*}_{y^{*}(t)},\, t\ge 0  \right)
  \end{equation*}
  is then uniformly integrable and we can let $t$ go to infinity in
  \eqref{eq_hawkes_core:eq_girsanov} to obtain
  \begin{equation*}
    \Ymap^{\#}\pi_{y}=\pi.
  \end{equation*}
  As a consequence,
  \begin{equation*}
    \pi=   \Ymap^{\#}\pi_{y}\sim \Ymap^{\#}\pi.
  \end{equation*}
  The proof is thus complete.
\end{proof}

\section{Invertibility}
\label{sec:invertibility}
\def\Yfrak{\mathbb Y} \def\Zfrak{\mathbb Z} We now define the notion of left and
right invertibility we will analyze in the following. We introduce a new
notation for maps from $\NN$ into itself as in full generality, such a map is
not necessarily induced by a time change $y$. We write in bold the
transformation when it is associated to time change and in blackboard bold for
an abstract transformation of $\NN$. Note that we must be careful when we
compose random variables: If $R$ and $\tilde{R}$ are random variable from
$\Omega$ to a space $E$ which are equal $\P$-a.s., we must ensure that for $S$
and $\tilde{S}$ which are $\P$-almost surely equal random variables from
$\Omega$ to $\Omega$, we still have
\begin{equation}
  \label{eq_hawkes_core:7}
  \P\Bigl( R\circ S\neq \tilde R\circ \tilde S\Bigr)=0.
\end{equation}
But,
\begin{equation*}
  \P\Bigl(R\circ S\neq \tilde R\circ \tilde S\Bigr)=\P\Bigl(R\circ S\neq \tilde R\circ  S \Bigr)=\P_{S}\Bigl(R\neq\tilde R\Bigr).
\end{equation*}
So a sufficient condition for \eqref{eq_hawkes_core:7} to hold is that
$\P_{S}\ll \P$.
\begin{definition}\label{def_hawkes_core:def_invertibility}
  Let~$(\NN,\mu,\, {\cal F}_{\infty})$ be a probability space and $\Yfrak$ a map
  from $\NN$ to $\NN$.

  The map $\Yfrak$ is left invertible if and only if $\Yfrak^{\#}\mu\ll \mu$ and
  there exists $\Zfrak\,:\, \NN\to \NN$ such that
  $\Zfrak\circ \Yfrak=\Id_{\NN}$, $\mu$-a.s.

  The map $\Yfrak$ is right invertible if and only if there exists
  $\Zfrak\, :\, \NN\to \NN$ such that $\Zfrak^{\#}\mu\ll \mu$ and
  $\Yfrak\circ \Zfrak=\Id_{\NN}$, $\mu$-a.s.

  The map $\Yfrak$ is invertible if it is both left and right invertible.
\end{definition}
\begin{lemma}\label{lem:preInversibilite}
  If there exists $\Zfrak$ such that $\Zfrak\circ \Yfrak=\Id_{\NN}$, $\mu$-a.s.
  then $\Yfrak\circ \Zfrak=\Id_{\NN}$, $\Yfrak^{\#}\mu$-a.s.

  If additionally, $\Yfrak^{\#}\mu$ is equivalent to $\mu$ and
  $\Zfrak^{\#}\mu\ll \mu$, then $\Yfrak$ is invertible and $\Zfrak^{\#}\mu$ is
  equivalent to $\mu$.
\end{lemma}
\begin{proof}
  We have
  \begin{align*}
    \Yfrak^{\#}\mu\Bigl( \Yfrak \circ \Zfrak=\Id_{\NN} \Bigr) & =   \mu\Bigl( \Yfrak \circ \Zfrak\circ \Yfrak=\Yfrak \Bigr) \\
                                                              & =\mu\Bigl( \Yfrak=\Yfrak \Bigr)                             \\
                                                              & =1.
  \end{align*}
  The first assertion follows. If the two measures $\Yfrak^{\#}\mu$ and $\mu$
  are equivalent, then $\Yfrak\circ\Zfrak=\Id_{\NN}$ $\mu$-almost-surely thus
  $\Yfrak$ is invertible.

  Let $A$ such that $\Zfrak^{\#}\mu(A)=0$. This means
  \begin{equation*}
    \espi{\mu}{\car_{A}\circ \Zfrak}=0.
  \end{equation*}
  Since $\Yfrak^{\#}\mu$ is equivalent to $\mu$, we get
  \begin{equation*}
    0=    \espi{\mu}{\car_{A}\circ \Zfrak\circ \Yfrak}=\espi{\mu}{\car_{A}}
  \end{equation*}
  hence $\mu\ll \Zfrak^{\#}\mu $ and the equivalence follows.
\end{proof}

We now state a technical lemma which indicates how the time changes of two
inverse maps are related to each other.
\begin{lemma}\label{lem:ConditionInversibilite}
  Let $\dot y\in \pred^{++}(\F)$ such that $\Ymap^{\#}\pi\ll \pi$. For $\dot z$
  a $\pi$-a.s. positive process such that $z^{*}(Y,t)$ is an
  $\F^{y^{*}}\!$-stopping time, we have
  \begin{equation}\label{eq_hawkes_brouillon:compositionChangementsTemps}
    ( \Zmap\circ \Ymap)(t)=N\Bigl(y^{*}\bigl(N,\,z^{*}(Y,t)\bigr)\Bigr).
  \end{equation}
  Then $\Zmap$ is the left inverse of $\Ymap$ if and only if
  \begin{equation}\label{eq_hawkes_brouillon:1}
    z^{*}\bigl(Y,\,t\bigr)=y(N,\,t), \ \forall t, \ \pi\text{-a.s.}
  \end{equation}
  or equivalently
  \begin{equation}\label{eq_hawkes_brouillon:29}
    z\bigl(Y,\,t\bigr)=y^{*}(N,\,t), \ \forall t, \ \pi\text{-a.s.}
  \end{equation}
\end{lemma}

\begin{proof}
  The first part comes from the identities:
  \begin{align*}
    ( \Zmap\circ \Ymap)(N)(t) & =Y\Bigl(z^{*}(Y,\,t)\Bigr)                     \\
                              & =N\Bigl(\ y^{*}\bigl(N,z^{*}(Y,t)\bigr)\Bigr).
  \end{align*}
  Thus,
  \begin{equation*}
    \Zmap\circ \Ymap=\Id_{\NN} \Longleftrightarrow y^{*}\bigl(N,\ z^{*}(Y,t)\bigr)=t, \ \forall t, \ \pi\text{-a.s.}
  \end{equation*}
  According to~\eqref{eq_hawkes_brouillon:3}, this
  entails~\eqref{eq_hawkes_brouillon:1}.
\end{proof}

We have already seen in Lemma~\ref{lem:EgaliteFiltration} that
${\cal N}^{y^{*}}$ is larger than $\mathcal{Y}$. Actually, there is a stronger
result:
\begin{theorem}\label{thm_hawkes_core:left_invertibility_filtration_equality}
  Let $\dot y\in \pred^{++}(\mathcal{N},\pi)$ such that $\Ymap^{\#}\pi\ll \pi.$
  Then, $\mathcal{N}^{y^{*}}=\mathcal{Y}$ if and only if $\ \Ymap$ admits a left
  inverse.
\end{theorem}
\begin{proof}
  If $\Ymap$ is left-invertible then there exists $z$ such that
  \eqref{eq_hawkes_brouillon:29} holds :
  \begin{equation*}
    z\bigl(Y,\,t\bigr)=y^{*}(N,\,t), \ \forall t, \ \pi\text{-a.s.}
  \end{equation*}
  This means that $y^{*}(N,s)$ is $\mathcal{Y}_{s}\subset\mathcal{Y}_{t} $
  measurable, thus
  $\mathcal{Y}_{t}\vee \sigma(y^{*}(s),\, s\le t)=\mathcal{Y}_{t}$.

  Conversely, if $\mathcal{N}^{y^{*}}_{t}=\mathcal{Y}_{t}$ then as $y^{*}(N,t)$
  is $ \mathcal{N}^{y^{*}}_{t}$ measurable, it is also
  $\mathcal{Y}_{t}$-measurable. Hence, for any $t\ge 0$, there exists a random
  variable $\tilde z(Y^{t},t)$ such that
  \begin{equation*}
    \tilde z(Y^{t},t)=y^{*}(N,t), \ \pi-\text{a.s.}
  \end{equation*}
  We can then find a full probability set $A$ such that
  \begin{equation}\label{eq_hawkes_brouillon:30}
    \tilde z(Y^{t}(\omega),t)=y^{*}(\omega,t), \forall t\in \Q, \forall \omega\in A.
  \end{equation}
  Let
  \begin{equation*}
    z(Y^{t}(\omega),t)=
    \begin{cases}
      y^{*}(\omega,t)                               & \text{ if } t\in \Q     \\
      \lim_{r_{n}\to t, r_{n}\in \Q}y^{*}(\omega,t) & \text{ if  }t\notin \Q.
    \end{cases}
  \end{equation*}
  By the sample path continuity of $y^{*}$, \eqref{eq_hawkes_brouillon:30} holds
  for any $t\in \R^{+}$ with probability~$1$. Hence,
  Lemma~\ref{lem:ConditionInversibilite} implies that $\Ymap$ admits a left
  inverse.
\end{proof}


We can now state the link between right invertibility of time change on $\NN$
and existence of g-Hawkes processes.
\begin{theorem}
  \label{thm_hawkes_brouillon:HawkesEquivautInverseDroite}
  Let $\dot y\in \pred^{++}(\mathcal{N})$. Then, $\Ymap$ is right invertible
  with a right inverse of the form $Z=N(z^{*})$ where for any $t> 0$,
  $z^{*}(N,t)$ is an $\mathcal{N}$-stopping time if and only if there exists a
  process $Z$, whose law is absolutely continuous with respect to $\pi$ and has
  compensator $y(Z,.)$, i.e. $Z$ is a solution of the equation
  \begin{equation}\label{eq_hawkes_brouillon:7}
    Z(t)=N\Bigl( y(Z,t) \Bigr)
  \end{equation}
  with the additional constraint that $y(Z,t)$ is an $\mathcal{N}$-stopping
  time.
\end{theorem}
\begin{proof}
  If $\Ymap$ is right invertible then $\Zmap^{\#}\pi\ll \pi$, hence it is
  meaningful to define the process $(y(Z(N),.), \, t\ge 0)$. If $Z=N(z^{*})$ is
  the right inverse of $\Ymap$, then, according
  to~\eqref{eq_hawkes_brouillon:29}, we have
  \begin{equation*}
    z^{*}\bigl(N,\,t\bigr)=y\left(Z(N),\,t\right),\ \pi-\text{a.s.}
  \end{equation*}
  so that, if we set
  \begin{equation*}
    Z(t):=N\Bigl(  z^{*}\left(N,\,t\right) \Bigr),
  \end{equation*}
  we get
  \begin{equation*}
    Z(t)= N\Bigl(  y\left(Z(N),\,t\right) \Bigr).
  \end{equation*}
  This means that $Z$ satisfies~\eqref{eq_hawkes_brouillon:7}. Since
  $z^{*}(N,t)$ is an $\mathcal{N}$-stopping time, so does $y(Z,t)$ and the
  additional constraint is altogether satisfied.

  Conversely, if $Z$ satisfies~\eqref{eq_hawkes_brouillon:7} and $y(Z,t)$ is an
  $\mathcal{N}$-stopping time for any $t\ge 0$, we set
  \begin{equation*}
    z^{*}(N,t)=y(Z,t)
  \end{equation*} so that $z^{*}(N,t)$ is a stopping time for any $t\ge 0$.
  Furthermore, according to \eqref{eq_hawkes_brouillon:1} we have
  \begin{math}
    \Ymap\circ \Zmap=\Id_{\NN},
  \end{math}
  which, with the hypothesis $\Zmap^{\#}\pi\ll \pi$, means that $\Ymap$ is right
  invertible.
\end{proof}
\begin{corollary}
  \label{lem_hawkes_core:egalité_tribu_hawkes}
  If $Z$ satisfies \eqref{eq_hawkes_brouillon:7} and $y(Z,t)$ is an
  $\mathcal{N}$ stopping time for any $t\ge 0$, then
  \begin{equation*}
    \mathcal{Z}_{t}=\mathcal{N}_{y(Z,t)}
  \end{equation*}
  where $\mathcal{Z}$ is the $\sigma$-field generated by the sample path of $Z$.
  Furthermore, $\Zmap^{\#}\pi=\pi_{y}.$
\end{corollary}
\begin{proof}
  Actually, Theorem~\ref{thm_hawkes_brouillon:HawkesEquivautInverseDroite}
  implies that $\Zmap$ is left invertible and the conclusion follows by Theorem
  \ref{thm_hawkes_core:left_invertibility_filtration_equality}. Consequently,
  according to \eqref{eq_hawkes_brouillon:7}, $Z$ has $y(Z,.)$ as a
  $(\pi,\mathcal{Z})$-compensator. This means that the law of $Z$ has the same
  compensator with respect to its minimal filtration as $N$ under $\pi_{y}$.
  Theorem~\ref{thm_hawkes_core:watanabe_charac} then entails that
  $\Zmap^{\#}\pi=\pi_{y}$.
\end{proof}
\begin{corollary}
  \label{lem_hawkes_core:representationmartingale}
  Let $\dot y\in \pred^{++}_{2}(\mathcal{N},\pi_{y})$. Assume that $\Ymap$ is
  right invertible and let $Z$ be the g-Hawkes process defined in
  Theorem~\ref{thm_hawkes_brouillon:HawkesEquivautInverseDroite}. Then, $Z$
  admits the martingale representation property: for any
  $F\in L^{2}(\NN\to \R,\Zmap^{\#}\pi)$, there exists $v$ a
  $\mathcal{Z}$-predictable process such that, $\pi$-a.s., we have
  \begin{equation*}
    F\circ \Zmap=\espi{\Zmap^{\#}\pi}{F}+\int_{0}^{\infty} v(Z,s)\left( \dif Z(s)-\dot y(Z,s)\dif s \right)
  \end{equation*}
  where
  \begin{equation*}
    v(N,s)=u\Bigl( Y,\, y(N,s) \Bigr).
  \end{equation*}
  Note that $v$ satisfies
  \begin{equation*}
    \espi{\pi}{\int_{0}^{\infty}v(Z,s)^{2}\ \dot y(Z,s)\dif s}<\infty.
  \end{equation*}
\end{corollary}
\begin{proof}
  Since $\Zmap$ is the right inverse of $\Ymap$, $\Ymap$ is the left inverse of
  $\Zmap$. According to Theorem
  \ref{thm_hawkes_core:left_invertibility_filtration_equality},
  $\mathcal{Z}=\mathcal{N}^{z^{*}}$. By the very definition of right
  invertibility, $\Zmap^{\#}\pi\ll \pi$ hence for $F$ a random variable on
  $(\NN,\pi)$, the random variable $F_{Z}:=F\circ \Zmap$ is well defined. When
  $F_{Z}$ is $\pi$-square integrable, we know from the martingale representation
  property for the Poisson process \cite{Jacod2013} that there exists a square
  integrable, predictable process $u$ such that
  \begin{equation*}
    F_{Z}=\espi{\pi}{F_{Z}}+\int_{0}^{\infty} u(s)\left( \dif N(s)-\dif s \right).
  \end{equation*}
  The relation \eqref{eq_hawkes_brouillon:29} 
  means that
  \begin{equation*}
    T_{k}(N)=y\bigl( Z,T_{k}(Z) \bigr).
  \end{equation*}
  We obtain
  \begin{align*}
    \int_{0}^{\infty}u(s)\dif N(s) & =\sum_{k\ge 1} u\bigl( N,\, T_{k}(N) \bigr)                         \\
                                   & =\sum_{k\ge 1}  u\left( Y(Z),\,  y\left( Z,T_{k}(Z) \right) \right) \\
                                   & =\int_{0}^{\infty} \Bigl(u(Y,.)\circ \Zmap \Bigr)(s)\dif Z(s).
  \end{align*}
  We know that $Z$ has $(\pi,\mathcal{N}^{z^{*}})$ compensator $Z^{p}(t)=y(Z,t)$
  hence this process is also the compensator of $Z$ for the filtration
  $\mathcal{Z}$. The change of variable $s=y(Z,r)$ yields
  \begin{align*}
    \int_{0}^{\infty}u(N,s)\dif s & =\int_{0}^{\infty}u\Bigl(N,\,y(Z,r)\Bigr)\dot y(Z,r)\dif r \\
                                  & =\int_{0}^{\infty}v(Z,s)\dif Z^{p}(r).
  \end{align*}
  The proof is thus complete.
\end{proof}
\begin{definition}
  For $\mu$ and $\nu$ two probability measures on $\NN$, the relative entropy of
  $\nu$ with respect to $\mu$ is given by
  \begin{equation*}
    H(\nu\/\mu)=
    \left\{
      \begin{aligned}
        \espi{\nu}{\log\left(\left. \frac{\D\nu}{\D\mu}\right|_{\mathcal{N}_{\infty}} \right)} & \text{ if } \nu\ll\mu \\
        +\infty                                                                                & \text{ otherwise.}
      \end{aligned}
    \right.
  \end{equation*}
\end{definition}
The next theorem is a crucial step towards the final proof. Roughly speaking,
given $\pi_{z}$, we aim to find a time change $y^{*}$ such that
$\Ymap^{\#}\pi=\pi_{z}$. The necessary condition is that $y^{*}$ solves
\eqref{eq_hawkes_brouillon:23} which implies
$\Lambda_{y}=1/\Lambda_{z}\circ \Ymap$. Following the proof of
\cite[Proposition~2.1]{LassalleStochasticinvertibilityrelated2012}, we have:
\begin{theorem}\label{thm_hawkes_brouillon:commeLassalle}
  Let $z\in \pred^{++}(\mathcal{N})$. Assume $\pi_{z}$ is equivalent to $\pi $
  on ${\cal N}_{\infty}$ and that there exists
  $\dot y\in\pred^{++}_{2}(\mathcal{N},\pi)$ such that
  $\Zmap\circ \Ymap=I_{\NN}$ $\pi$-a.s. Then the following assertions are
  equivalent:
  \begin{enumerate}[label=({\roman*}),ref=({\roman*})]
    \item\label{item:1} We have
          \begin{equation}
            \label{eq_hawkes_brouillon:35}
            \espi{\pi}{\Lambda_{y}}=1 \text{ and } \pi(\Lambda_{y}>0)=1.
          \end{equation}
    \item\label{item:2} $\Ymap^{\#}\pi$ is equivalent to $\pi$ {on
          ${\mathcal N}_{\infty}$}.
    \item\label{item:3} $\Zmap$ is left invertible with inverse $\Ymap$.
    \item\label{item:4} $\Ymap^{\#}\pi=\pi_{z}$ {on ${\mathcal N}_{\infty}$} and
          we have the following identity
          \begin{equation}
            \label{eq:egaliteFondamentale}
            \log \bigl(\Lambda_{z}\circ\Ymap\bigr)=\int_{0}^{\infty} \log\bigl(\dot y^{*}(N,\,s)\bigr)\dif Y(s)+\int_{0}^{\infty}\bigl(1-\dot y^{*}(N,\,s)\bigr)\dif s, \ \pi-\text{a.s.}
          \end{equation}
  \end{enumerate}
\end{theorem}
\begin{proof}
  \noindent \ref{item:1}$\Longrightarrow$\ref{item:2}. According to
  Lemma~\ref{thm_hawkes_core:Equivalence_criterion}, $\pi_{y}$ is equivalent to
  $\pi$ on $\mathcal{N}_{\infty}$. Apply
  Theorem~\ref{thm_hawkes_brouillon:AbsolueContinuiteChangementTemps} and point
  \ref{item:2} follows.

  \noindent \ref{item:2}$\Longrightarrow$\ref{item:3}. 
  According to Lemma~\ref{lem:preInversibilite}, $\Ymap \circ \Zmap=\Id_{\NN}$,
  $\Ymap^{\#}\pi$ a.s. In view of \ref{item:2}, this identity holds $\pi$ almost
  surely. Moreover, still according to~\ref{item:2},
  \begin{equation*}
    \Zmap^{\#}\pi\sim \Zmap^{\#}(\Ymap^{\#}\pi)=\pi
  \end{equation*}
  and \ref{item:3} follows from Lemma~\ref{lem:preInversibilite}.

  \noindent \ref{item:3}$\Longrightarrow$\ref{item:4}. It follows
  from~\ref{item:3} that
  $\pi_z=(\Ymap\circ \Zmap)^{\#}\pi_z=\Ymap^{\#}(\Zmap^{\#}\pi_z)$. Recall that
  the quasi-invariance theorem says that $\Zmap^{\#}\pi_z=\pi$ hence
  $\Ymap^{\#}\pi=\pi_z$.

  According to \eqref{eq_hawkes_brouillon:6b} and \eqref{eq:chgttempsIntStob},
  \begin{equation*}
    \Lambda^{*}_{z^{*}}\circ \Ymap\\=\exp\left( -\int_{0}^{\infty }\log \Bigl(   \dot
      z^{*}\bigl(Y,\,y^{*}(N,s) \bigr) \Bigr)\dif Y(s)+\int_{0}^{\infty} \Bigl( \dot
      z^{*}\bigl(Y,\, y^{*}(N,s)\bigr)-1 \Bigr)\,\dot y^{*}(N,s)\dif s\right).
  \end{equation*}
  Since $\Zmap\circ \Ymap=\Id_{\NN}$, according to \eqref{eq_hawkes_brouillon:1}
  we have
  \begin{equation*}
    z^{*}\bigl(Y,y^{*}(N,t)\bigr)=t.
  \end{equation*}
  By differentiation, we obtain
  \begin{equation}
    \label{eq_hawkes_brouillon:23}    \dot z^{*}\bigl( Y, y^{*}(N,t) \bigr)\times\, \dot y ^{*}(N,t)=1,\, \pi-\text{a.s.}
  \end{equation}
  We obtain
  \begin{equation*}
    \log  (\Lambda_{z}\circ \Ymap)=\int_{0}^{\infty }\log \Bigl(   \dot
    y^{*}\bigl(N,s \bigr) \Bigr)\dif Y(s)+\int_{0}^{\infty} \Bigl( \frac{1}{\dot y^{*}(N,s)}-1 \Bigr)\,\dot y^{*}(N,s)\dif s
  \end{equation*}
  and \eqref{eq:egaliteFondamentale} follows 

  \noindent \ref{item:4}$\Longrightarrow$\ref{item:1}. Recall that
  \begin{equation*}
    \Lambda_{y}=\exp\left(- \int_{0}^{\infty} \log\bigl(\dot y^{*}(N,s)\bigr)\dif Y(s)-\int_{0}^{\infty}\bigl(1-\dot y^{*}(N,s)\bigr)\dif s \right).
  \end{equation*}
  Equation~\eqref{eq:egaliteFondamentale} amounts to
  \begin{align*} {\Lambda_{y}}= {\frac{1}{\Lambda^{*}_{z^{*}}}\circ \Ymap}=\frac{1}{\Lambda_{z}}\circ \Ymap.
  \end{align*}
  Since $\Ymap^{\#}\pi=\pi_{z}$, we obtain
  \begin{align*}
    \espi{\pi}{\Lambda_{y}}
    & =\espi{\pi_{z}}{\frac{1}{\Lambda_{z}}}             \\
    & =\espi{\pi}{\frac{1}{\Lambda_{z}}\ \Lambda_{z}}=1.
  \end{align*}
  Moreover, still from ~\eqref{eq:egaliteFondamentale}, we deduce that
  \begin{align*}
    \pi\left( \Lambda_{y}=0 \right) & =\pi\left( \frac{1}{\Lambda_{z}}\circ \Ymap=0 \right) \\
                                    & =\Ymap^{\#}\pi\left( \Lambda_{z}=+\infty \right)      \\
                                    & =\pi_{z}\left( \Lambda_{z}=+\infty \right).
  \end{align*}
  According to the hypothesis, $\pi_{z}$ is equivalent to $\pi$ and
  Lemma~\ref{thm_hawkes_core:Equivalence_criterion} entails that
  $\espi{\pi}{\Lambda_{z}}$ is finite hence
  \begin{math}
    \pi\left( \Lambda_{z}=+\infty \right)=0 \text{ and then
    } \pi_{z}\left( \Lambda_{z}=+\infty \right)=0.
  \end{math}
  It follows that
  \begin{equation*}
    \pi\left( \Lambda_{y}=0 \right)=0,
  \end{equation*}
  so that \ref{item:1} holds.
\end{proof}
\begin{lemma}
  \label{lem_hawkes_core:2}
  Let $\dot y \in \pred^{++}_{2}({\cal N},\pi_{y})$ such that
  $\espi{\pi}{\Lambda_{y}}=1$ and
  \begin{equation*}
    \espi{\pi}{\int_{0}^{\infty}\m\left( \dot y^{*}(N,s)-1 \right)\dif s}<\infty.
  \end{equation*}
  Then,
  \begin{equation}
    \label{eq_hawkes_core:23}
    \espi{\pi}{-\log \Lambda^{*}_{y^{*}}}\le \espi{\pi}{\int_{0}^{\infty}\m\bigl(\dot y(N,s)-1\bigr)\dif s.}
  \end{equation}
\end{lemma}
\begin{proof}
  We have already seen that
  \begin{equation*}
    \dot y^{*}(t)=\frac{1}{\dot y\bigl(y^{*}(t)\bigr)}=\tau_{y^{*}}\left(\frac{1}{\dot y}\right)(t).
  \end{equation*}
  By hypothesis, $\dot y $ is $\mathcal{N}$-predictable, thus, according to
  \cite[Theorem 10.17(c)]{JacodCalculstochastiqueproblemes1979}, $\dot y^{*}$ is
  predictable with respect to the filtration $\mathcal{N}^{y^{*}}$.

  From \eqref{eq_hawkes_core:4}, we have
  \begin{align*}
    \esp{-\log \Lambda_{y^{*}}^{*}(t)} & =\esp{\int_{0}^{t}\log\bigl(\dot y^{*}(s)\bigr)\dif Y(s)+\int_{0}^{t}\bigl(1-\dot y^{*}(s)\bigr)\dif s} \\
                                       & =\esp{\int_{0}^{t}\dot y^{*}(s)\log\bigl(\dot y^{*}(s)\bigr)+1-\dot y^{*}(s)\dif s}                     \\
                                       & \le \esp{\int_{0}^{\infty}\m\bigl(\dot y^{*}(s)-1\bigr)\dif s},
  \end{align*}
  since $Y$ has $\mathcal{N}^{y^{*}}$ compensator $y^{*}$ and $\log \dot y^{*}$
  is $\mathcal{N}^{y^{*}}$ predictable. It remains to prove that we can pass to
  the limit in the left-hand-side. Consider the non-negative, convex function
  $\psi(x)=x-\log x$. From Fatou's Lemma, we have
  \begin{align*}
    \esp{\psi(\Lambda_{y^{*}}^{*})} & \le \liminf_{t\to \infty}  \  \esp{\psi(\Lambda_{y^{*}}^{*}(t))}   \\
                                    & \le 1+ \esp{\int_{0}^{\infty}\m\bigl(\dot y^{*}(s)-1\bigr)\dif s}.
  \end{align*}
  This means that the non-negative submartingale
  $(\psi(\Lambda_{y^{*}}^{*}(t)),\, t\ge 0)$ is uniformly integrable. From
  \eqref{eq_hawkes_core:1}, we know that $\Lambda_{y}=\Lambda_{y^{*}}^{*}$ hence
  in view of the first hypothesis, the non-negative martingale
  $(\Lambda_{y^{*}}^{*}(t),\, t\ge 0)$ is uniformly integrable,
  \begin{equation*}
    -\log \Lambda_{y^{*}}^{*}(t) \xrightarrow[t\to\infty]{L^{1} } -\log \Lambda_{y^{*}}^{*}.
  \end{equation*}
  This means that
  \begin{equation*}
    1+ \esp{-\log \Lambda_{y^{*}}^{*}}\le 1+ \esp{\int_{0}^{\infty}\m\bigl(\dot y^{*}(s)-1\bigr)\dif s}.
  \end{equation*}
  The proof is thus complete.
\end{proof}
We arrive now at the main result of this section, the entropic criterion of
(left) invertibility. Before going into the details of the proof, we explain its
main idea. At the very beginning, we are given $y$ an increasing predictable
process whose inverse can be used as a time change. From $y$, one can construct
$\Lambda_{y}$ and $\Ymap^{\#}\pi$. Note carefully that $\Lambda_{y}$ is not the
density of $\Ymap^{\#}\pi$ with respect to $\pi$. Actually, under the hypothesis
we made on $y$, $\Ymap^{\#}\pi$ is absolutely continuous with respect to $\pi$
and there exists $z$ such that
\begin{equation*}
  \frac{\D \Ymap^{\#}\pi}{\D \pi}=\Lambda_{z}.
\end{equation*}
If we have
\begin{equation}
  \label{eq_hawkes_core:15}
  \log \Lambda_{z}\circ \Ymap=-\log \Lambda_{y},
\end{equation}
then, from the representations of these quantities (see
Theorem~\ref{thm_hawkes_core:compositionTimeChange}), we see that
\begin{equation*}
  z^{*}\left(Y, y^{*}(N,t)\right)=1,
\end{equation*}
which according to Lemma \ref{lem:ConditionInversibilite} means that
$\Zmap\circ \Ymap=\Id_{\NN}$. Before going further, recall that at the terminal
time $\Lambda_{y}=\Lambda_{y ^{*}}^{*}$ but at time $t$, $\Lambda_{y}(t)$ is
$\mathcal{N}_{t}$ measurable whereas $\Lambda_{y ^{*}}^{*}(t)$ is
$\mathcal{N}_{y^{*}(N,t)}$ adapted. For the sake of clarity, we write all the
conditions in terms of $\Lambda_{y}$ even if in the following proofs we need to
use $\Lambda^{*}_{y^{*}}$.

Now, on the one hand, quasi-invariance and Fatou lemma induce that
\begin{equation}
  \label{eq_hawkes_core:16}
  \log \Lambda_{z}\circ \Ymap\le -\log \espi{\pi}{\Lambda_{y}\/ \mathcal{Y}_{\infty}}.
\end{equation}
On the other hand, classical computations and Jensen inequality show that
\begin{multline}
  \label{eq_hawkes_core:9}
  H(\Ymap^{\#}\pi\/ \pi)=\espi{\pi}{\log \Lambda_{z}\circ \Ymap}\le -\espi{\pi}{\log \espi{\pi}{\Lambda_{y}\/\mathcal{Y}_{\infty}}}\\\le -\espi{\pi}{\log \Lambda_y}=\espi{\pi}{\int_{0}^{\infty}\m(\dot y^{*}(s)-1)\dif s}.
\end{multline}
If the entropic criterion is satisfied, then all the inequalities of
\eqref{eq_hawkes_core:9} are indeed transformed into equalities. The equality
condition in the conditional Jensen inequality implies that
\begin{equation}
  \label{eq_hawkes_core:18}
  \espi{\pi}{\Lambda_{y}\/ \mathcal{Y}_{\infty}}=\Lambda_{y},
\end{equation}
which in turn entails that \eqref{eq_hawkes_core:16} becomes
\begin{equation*}
  \log \Lambda_{z}\circ \Ymap\le -\log \Lambda_{y}.
\end{equation*}
In view of the first inequality of \eqref{eq_hawkes_core:9}, we obtain the
identity \eqref{eq_hawkes_core:15} and the invertibility follows.

Remark that \eqref{eq_hawkes_core:18} says that $\Lambda_{y}$, which is a priori
$\mathcal{N}_{\infty}=\mathcal{N}^{y^{*}}_{\infty}$ measurable, is actually
$\mathcal{Y}_{\infty}$ measurable which in view of Theorem
\ref{thm_hawkes_core:left_invertibility_filtration_equality}, is a necessary
condition for $\Ymap$ to be left-invertible.
\begin{theorem}\label{thm_hawkes_core:left_invertible_entropy}
  Let $\dot y \in \pred^{++}_{2}({\cal N},\pi_{y})$ such that
  $\espi{\pi}{\Lambda_{y}}=1$ and
  \begin{equation*}
    \espi{\pi}{\int_{0}^{\infty}\m\left( \dot y^{*}(s)-1 \right)\dif s}<\infty.
  \end{equation*}
  If $\Ymap^{\#}\pi\ll \pi$, we have
  \begin{equation}\label{eq_hawkes_core:pre_entropy}
    H( \Ymap^{\# }\pi\/\pi)\leq {\mathbf E}\left[\int_0^{\infty} \m\bigl(\dot{y}^*(s) -1\bigr) \dif s\right].
  \end{equation}
  {Moreover}, the map $\Ymap$ is left invertible
  if and only if
  \begin{equation}\label{eq_hawkes_core:11}
    H(\Ymap^{\#}\pi\/ \pi)=\espi{\pi}{\int_{0}^{\infty}\m\left( \dot y^{*}(s)-1 \right)\dif s}.
  \end{equation}
\end{theorem}
\begin{remark}
  Following \cite{Atar2012}, the function $x\mapsto \m(x-1)$ plays in the
  Poisson settings the exact same role as the function $x\mapsto x^{2}/2$ in the
  Gaussian world. Since the entropic criterion in the Wiener reads as
  \begin{equation*}
    H( U^{\#}\mu\/\mu)= \frac{1}{2}\espi{\mu}{\int_{0}^{1}\dot u(s)^{2}\dif s}
  \end{equation*}
  where $U$ is defined in \eqref{eq_hawkes_brouillon:8}, we see that the
  Gaussian-Poisson analogy alluded to in \cite{Atar2012} goes even further.
\end{remark}

\begin{proof}
  Since $\Ymap^{\#}\pi\ll \pi$,
  Theorem~\ref{thm_hawkes_brouillon:absolueContinuite} implies that there exists
  $\dot z\in \pred^{+}_{2}({\cal N},\ \Ymap^{\#}\pi)$ such that
  \begin{equation*}
    \left.\frac{\D \Ymap^{\#}\pi}{\D\pi}\right|_{{\cal N}_{\infty}}=\Lambda_{z}.
  \end{equation*}
  The quasi-invariance Theorem~\ref{thm_hawkes_brouillon:Pregirsanov} then says
  that for each $t>0$, for $f\, :\, \NN\to \R^{{+}}$ bounded and continuous,
  \begin{align*}
    \espi{\pi}{f\circ \Ymap^{t}} & =    \espi{\pi}{f(Y^{t})}                                                               \\
                                 & =  \espi{\pi}{f \Lambda_{z}(t)}                                                         \\
                                 & =\espi{\pi}{f\circ \Ymap^{t} \ \Lambda_{z}\circ \Ymap^{t}(t)\  \Lambda_{y^{*}}^{*}(t)}.
  \end{align*}
  By Fatou's Lemma, we obtain
  \begin{equation*}
    \espi{\pi}{f\circ \Ymap \  \Lambda_{z}\circ \Ymap\  \Lambda_{y^{*}}^{*}}\le     \espi{\pi}{f\circ \Ymap}
  \end{equation*}
  or equivalently
  \begin{equation*}
    \espi{\pi}{f\circ \Ymap \  \Lambda_{z}\circ \Ymap\ \esp{ \Lambda_{y^{*}}^{*}\/\mathcal{Y}_{\infty}}}\le     \espi{\pi}{f\circ \Ymap}
  \end{equation*}
  Hence, $\pi$-a.s. we have
  \begin{equation}
    \label{eq_hawkes_core:33}\Lambda_{z}\circ \Ymap\times  \espi{\pi}{\Lambda_{y^{*}}^{*}\left|  {\cal Y}_{\infty}\right.}\le 1.
  \end{equation}
  It follows that
  \begin{align}
    0\le        H(\Ymap^{\#}\pi \/ \pi) & =\espi{\pi}{\log \Lambda_{z}\circ \Ymap}\notag                                                                  \\
                                        & \le -\espi{\pi}{\log \espi{\pi}{\Lambda_{y^{*}}^{*}\left|  {\cal Y}_{\infty}\right.}}.\label{eq_hawkes_core:19}
  \end{align}

  Since $-\log$ is convex, the Jensen inequality stands that

	\begin{align}
	  H(\Ymap^{\#}\pi \/ \pi) & \le -\espi{\pi}{\log \Lambda_{y^{*}}^{*}}\label{eq:Jensen1}               \\
                            & =\espi{\pi}{\int_{0}^{\infty}\m\left(\dot y^{*}(s)-1\right)\dif s}\notag,
	\end{align}
	according to Lemma~\ref{lem_hawkes_core:2}. Then the first part holds.

	Assume now that \eqref{eq_hawkes_core:11} holds. Then
  \eqref{eq_hawkes_core:19} and \eqref{eq:Jensen1} are indeed equalities. On the
  one hand, this means that we have equality in the Jensen inequality used to
  derive~\eqref{eq:Jensen1}. Since $-\log$ is strictly convex, it follows that
  \cite[Cor. 2.1]{Mussmann1988}:
	\begin{equation*}
	  \espi{\pi}{\Lambda_{y^{*}}^{*}\left|  {\cal Y}_{\infty}\right.}= \Lambda_{y^{*}}^{*},\ \pi-\text{a.s.}
	\end{equation*}

	On the other hand, this also implies that \eqref{eq_hawkes_core:19} is an
  equality and as a consequence of \eqref{eq_hawkes_core:33}, we obtain
	\begin{equation}\label{eq_hawkes_core:12}
	  \log \Lambda_{z^{*}}^{*}\circ \Ymap=-\log \Lambda_{y^{*}}^{*}.
	\end{equation}
	According to \eqref{eq_hawkes_brouillon:6b} and \eqref{eq:chgttempsIntStob},
	\begin{equation*}
	  \log     \Lambda^{*}_{z^{*}}\circ \Ymap\\= -\int_{0}^{\infty }\log \Bigl(   \dot
	  z^{*}\bigl(Y,\,y^{*}(N,s) \bigr) \Bigr)\dif Y(s)+\int_{0}^{\infty} \Bigl( \dot
	  z^{*}\bigl(Y,\, y^{*}(N,s)\bigr)-1 \Bigr)\,\dot y^{*}(N,s)\dif s
	\end{equation*}
	and
	\begin{equation*}
	  -\log \Lambda_{y^{*}}^{*}=\int_{0}^{\infty} \log \left(\dot y^{*}(N,\,s)  \right)\dif Y(s)-\int_{0}^{\infty} \left( \dot y^{*}(N,\,s)-1 \right)\dif s.
	\end{equation*}
	Hence, according to \eqref{eq_hawkes_core:11}, we can take the conditional
  expectation with respect to $\mathcal{Y}_{t}$ in both equalities to obtain:
	\begin{multline*}
	  -\int_{0}^{t}\log \Bigl(   \dot
	  z^{*}\bigl(Y,\,y^{*}(N,s) \bigr) \Bigr)\dif Y(s)+\int_{0}^{t} \Bigl( \dot
	  z^{*}\bigl(Y,\, y^{*}(N,s)\bigr)-1 \Bigr)\,\dot y^{*}(N,s)\dif s\\
	  =    \int_{0}^{t} \log \left(\dot y^{*}(N,\,s)  \right)\dif Y(s)-\int_{0}^{t} \left( \dot y^{*}(N,\,s)-1 \right)\dif s.
	\end{multline*}
	Equating the jumps yields
	\begin{equation*}
	  \dot  z^{*}\bigl(Y,\, y^{*}(N,s)\bigr)\times\dot y^{*}(N,\,s)=1,
	\end{equation*}
	then, by integration, we get
	\begin{equation*}
	  z^{*}\left(Y,y^{*}(N,t)\right)=t.
	\end{equation*}
	Furthermore,
	\begin{equation*}
	  \left( z^{*}(Y,t)\le s \right)=\left(y(N,t)\le s  \right)=\left( y^{*}(N,s)\ge t \right)\in {\cal N}^{y^{*}}_{s},
	\end{equation*}
	hence $z^{*}(Y,t)$ is an ${\cal N}^{y^{*}}$ stopping time. According to
  \eqref{eq_hawkes_brouillon:1}, this means that $\Zmap\circ \Ymap=\Id_{\NN}$
  and then $\Ymap$ is left invertible.

	Conversely, if $\Ymap$ is left invertible. According to the
  Definition~\ref{def_hawkes_core:def_invertibility}, $\Ymap^{\#}\pi$ is
  absolutely continuous with respect to~$\pi$.
	Let us denote by $\Zfrak $ the map such that
	\begin{equation*}
	  \Zfrak \circ \Ymap=\Id_{\NN},\ \pi-\text{a.s.}
	\end{equation*}
	Define the process $z$ by
	\begin{equation*}
	  z\left( N,t \right)=y^{*}\bigl(\Zfrak(N),\, t\bigr), \forall t\ge 0.
	\end{equation*}
	Since $\Ymap^{\#}\pi\ll \pi$,
	\begin{equation*}
	  \Ymap^{\#}\pi\left(  z\left( N,. \right)=y^{*}\left(\Zfrak(N),\, .\right) \right)=1
	\end{equation*}
	and
	\begin{align*}
	  \pi\left( z(Y,.)=y^{*}(N,.) \right)
	  & = \Ymap^{\#}\pi\left( z(N,.)=y^{*}(\Zfrak(N),.) \right).
	\end{align*}
	This means that
	\begin{equation*}
	  z(Y,t)=y^{*}(N,t),\  \forall t\ge 0, \, \pi\text{-a.s.}
	\end{equation*}
	or otherwise stated, $\Zmap\circ \Ymap=\Id_{\NN}$, $\pi$-a.s. Moreover, by
  differentiation, we get
	\begin{equation}
	  \label{eq_hawkes_core:13}
	  \dot z^{*}(Y,\, y^{*}(N,t))=\frac{1}{\dot y^{*}(N,t)},\ \forall t\ge 0, \, \pi-\text{a.s.}
	\end{equation}
	Since $\dot y(N,.)$ belongs to $\pred^{++}_{2}({\cal N},\pi)$, so does
  $\dot z^{*}(Y,.)$, hence we can write
	\begin{equation*}
	  \log \Lambda_{z^{*}}^{*}\circ \Ymap=\\ -\int_{0}^{\infty }\log \Bigl(   \dot
	  z^{*}\bigl(Y,\,y^{*}(N,s) \bigr) \Bigr)\dif Y(s)+\int_{0}^{\infty} \Bigl( \dot
	  z^{*}\bigl(Y,\, y^{*}(N,s)\bigr)-1 \Bigr)\,\dot y^{*}(N,s)\dif s,
	\end{equation*}
	and \eqref{eq_hawkes_core:13} entails that
	\begin{equation}
	  \label{eq_hawkes_core:14}
	  \log \Lambda_{z^{*}}^{*}\circ \Ymap= -\log \Lambda_{y^{*}}^{*}.
	\end{equation}
	Let
	\begin{equation*}
	  R=\frac{\D \Ymap^{\#}\pi}{\D\pi}\cdotp 
	\end{equation*}
	For any $f\, :\, \NN\to \R$ continuous and bounded, for any $t>0$, we have
	\begin{align*}
	  \espi{\pi}{f R} & =\espi{\pi}{f\circ \Ymap}                                             \\
                    & =\espi{\pi}{(f\Lambda_{z^{*}}^{*} )\circ \Ymap \ \Lambda_{y^{*}}^{*}} \\
                    & =\espi{\pi}{f\Lambda_{z^{*}}^{*} }
	\end{align*}
	according to \eqref{eq_hawkes_core:14} and to the quasi-invariance Theorem. It
  follows that
	\begin{math}
	  R= \Lambda_{z^{*}}^{*},\ \pi-\text{a.s.}
	\end{math}
	Plug this identity into \eqref{eq_hawkes_core:14} to obtain
	\begin{align*}
	  H(\Ymap^{*}\pi\/\pi) & =\espi{\pi}{\log R\circ \Ymap}                                      \\
                         & =\espi{\pi}{\log \Lambda_{z^{*}}^{*}\circ \Ymap}                    \\
                         & =\espi{\pi}{-\log \Lambda_{y^{*}}^{*}}                              \\
                         & =\espi{\pi}{\int_{0}^{\infty}\m\left(\dot y^{*}(s)-1\right)\dif s}.
	\end{align*}
	The entropic criterion is thus satisfied.
\end{proof}
\section{Variational representation of the entropy}
\label{sec:vari-repr-entr}
We now give an interesting application of the previous considerations where the
entropic criterion is the key to the approximation procedure needed in the proof
of this representation of the entropy.

Let
\begin{align*}
	\pred^{++}_{\m}                          & =\left\{y,\, \dot y \in \pred^{++}(\mathcal{N}) \text{ and } (\dot y -1)\in L^{1}(\NN; \L_{\m},\pi)\right\} \\
	\pred^{++}_{\infty,\pc}(\mathcal{N},\pi) & = \pred^{++}_{\infty}(\mathcal{N},\pi) \cap \{y, \dot y\text{ piecewise constant}\}                         \\
	\mathfrak M_{\m}(\NN)                    & =\left\{ \mu,\ \exists y\in \pred^{++}_{\m} \text{ such that } \mu=\Ymap^{\#}\pi \right\}.
\end{align*}
The first step of the proof consists in proving the existence of a g-Hawkes
process for a piece-wise constant time change (see
\cite{hartmann2016variational,Stroock2006} for the Brownian analog).
\begin{lemma}
	\label{thm_hawkes_brouillon:ChgtTempsPiecewiseInversible} Let
  $\dot y\in \pred^{++}_{\infty}(\mathcal{N},\pi)$ be piecewise constant: If we
  denote by $T$ the time after which $\dot y(N,s)=1$, consider a partition of
  $[0,T]$, $0=t_{0}<t_{1}<\ldots<t_{k}=T<t_{k+1}=+\infty$ and assume that there
  exist $\alpha_{0}\in \R^{+}$ and some random variables
  $(\alpha_{j}, j=1,\cdots,k)$ such that for some $\epsilon>0$,
	\begin{equation*}
	  \epsilon \le \alpha_{j}(N) \le 1/\epsilon,\  \forall j=0,\cdots,k
	\end{equation*}
	and
	\begin{equation*}
	  \dot y(N,s)=\alpha_{0}\car_{(0,t_{1}]}(t)+\sum_{j=1}^{k-1}\alpha_{j}(N^{t_{j}})\car_{(t_{j},t_{j+1}]}(s)+\alpha_{k}(N^{T}) \car_{[T,\infty)}(s).
	\end{equation*}
	Then, $\Ymap$ is invertible.
\end{lemma}
\begin{proof}
	We first prove that $\Ymap$ is right invertible. The g-Hawkes process $Z$ is
  constructed inductively. On $[0,t_{1}]$, we set
	\begin{equation*}
	  Z(t)=N(\alpha_{0}t).
	\end{equation*}
	Then,
	\begin{equation*}
	  y(Z,t)=\alpha_{0}t
	\end{equation*}
	and we do have $Z(t)=N\bigl( y(Z,t) \bigr)$. For any $t\le t_{1}$, $y(N,t)$ is
  deterministic hence it is an $\mathcal{N}$-stopping time.

	Assume that $Z$ is constructed on $[0,t_{m}]$ with $m<k$ and $y(Z,t)$ is an
  $\mathcal{N}$-stopping time for $t\le t_{m}$. For $t\in[t_{m},t_{m+1}]$, we
  have
	\begin{equation}\label{eq_hawkes_brouillon:40}
	  y(Z,t)=y(Z,t_{m})+\alpha_{m}(Z^{t_{m}})(t-t_{m}).
	\end{equation}
	By the induction hypothesis, $y(Z,t_{m})$ is an $\mathcal{N}$-stopping time
  hence the $\sigma$-field $\mathcal{N}_{y(Z,t_{m})}$ is well defined and
  $y(Z,t_{m})$ belongs to this $\sigma$-field. Furthermore, for $t\le t_{m}$,
	\begin{equation*}
	  Z(t)=N\Bigl( y(Z,t) \Bigr)\in \mathcal{N}_{y(Z,t_{m})}.
	\end{equation*}

	It follows that for $t\in[t_{m},t_{m+1}]$,
	\begin{equation*}
	  y(Z,t)\in \mathcal{N}_{y(Z,t_{m})}.
	\end{equation*}
	Since $\alpha_{m}>0$, for any $s\ge 0$, according to
  \eqref{eq_hawkes_brouillon:40},
	\begin{align*}
	  \Bigl( y(Z,t)\le s \Bigr) & =\Bigl( \alpha_{m}(Z^{t_{m}})\le  \frac{s-y(t_{m})}{t-t_{m}} \Bigr)                       \\
                              & =\Bigl( \alpha_{m}(Z^{t_{m}})\le  \frac{s-y(t_{m})}{t-t_{m}} \Bigr)\cap (y(Z,t_{m})\le s)
	\end{align*}
	which belongs to $\mathcal{N}_{s}$ by the definition of
  $\mathcal{N}_{y(Z,t_{m})}$. Hence $y(Z,t)$ is an $\mathcal{N}$-stopping time
  for $t\le t_{m+1}$. Moreover, \eqref{eq_hawkes_brouillon:40} guarantees that
  there is no ambiguity to define $Z$ on $[t_{m},t_{m+1}]$ by $Z(t)=N(y(Z,t))$.
  According to Theorem~\ref{thm_hawkes_brouillon:HawkesEquivautInverseDroite},
  $\Ymap$ is right invertible.

	Remark that $Z$ has $(\mathcal{N},\pi)$ compensator $y(Z,.)$. In view of
  Theorem~\ref{thm_hawkes_brouillon:absolueContinuite},
	\begin{equation*}
	  \left. \frac{\D \pi_{z}}{\D \pi}\right|_{\mathcal{N}_{t}}=\Lambda_{y(Z,t)}.
	\end{equation*}
	From the form of $y$, it is clear that $\Lambda_{y(Z,t)}>0$ for all $t\ge 0$
  and that
	\begin{equation*}
	  \Bigl( \Lambda_{y(Z,t)},\, t\ge 0 \Bigr)
	\end{equation*}
	is uniformly integrable. Point \ref{item:1} of
  Theorem~\ref{thm_hawkes_brouillon:commeLassalle} then entails that $\Ymap$ is
  left invertible and then invertible.
\end{proof}
The theorem reads as follows:
\begin{theorem}[Variational representation of the entropy]
  \label{thm_hawkes_core:variational_representation}
	Let $f\, :\, \NN\to \R$ such that
	\begin{equation*}
	  \espi{\pi}{|f|(1+e^{f})}<\infty.
	\end{equation*}
	Then,
	\begin{equation*}
	  \label{eq_hawkes_brouillon:39}
	  \log \espi{\pi}{e^{f}}=\sup_{y\in \pred^{++}_{\m}}\left( \espi{\pi}{f(N\circ y^{*})}-\espi{\pi}{\int_{0}^{\infty}\m(\dot y^{*}(s)-1)\dif s} \right).
	\end{equation*}
\end{theorem}
\begin{proof}
	The duality between the relative entropy and the logarithmic Laplace transform
  says that
	\begin{equation}\label{eq_hawkes_brouillon:41}
	  \log \espi{\pi}{e^{f}}=\sup_{\mu\in \MM(\NN)}\left( \int_{\NN}f\dif \mu  -H(\mu \,|\, \pi)\right)
	\end{equation}
	where $\MM(\NN)$ is the set of probability measures on $\NN$ which are
  absolutely continuous with respect to~$\pi$ on~$\mathcal{N}_{\infty}$.
  Furthermore, the supremum is attained at the measure $\mu_{f}$ whose
  $\pi$-density is given by
	\begin{equation*}
	  \frac{\D \mu_{f}}{\D \pi}=\frac{e^{f}}{\espi{\pi}{e^{f}}}\cdotp
	\end{equation*}
	In view of \eqref{eq_hawkes_brouillon:41}, we evidently have
	\begin{equation*}
	  \log \espi{\pi}{e^{f}}\ge \sup_{\mu\in \mathfrak M _{\m}(\NN)}\left( \int_{\NN}f\dif \mu  -H(\mu \,|\, \pi)\right).
	\end{equation*}
	According to Lemma~\ref{thm_hawkes_brouillon:ChgtTempsPiecewiseInversible},
  for $y\in \pred^{++}_{\m} $
	\begin{equation*}
	  \int_{\NN}f\dif \mu  -H(\Ymap^{\#}\pi \,|\, \pi)\ge \espi{\pi}{f(N\circ y^{*})}-\espi{\pi}{\int_{0}^{\infty} \m(\dot y^{*}(s)-1)\dif s}.
	\end{equation*}
	Since $ \pred^{++}_{\infty,\pc}(\mathcal{N},\pi) \subset \pred^{++}_{\m}$, the
  entropic criterion implies that
	\begin{align*}
	  \sup_{\mu\in \MM(\NN)}  \int_{\NN}f\dif \mu  -H(\mu \,|\, \pi) & \ge \sup_{y\in \pred^{++}_{\m}}\espi{\pi}{f(N\circ y^{*})}-\espi{\pi}{\int_{0}^{\infty} \m(\dot y^{*}(s)-1)\dif s} \\
                                                                   & \ge     \sup_{y\in \pred^{++}_{\infty,\pc}(\mathcal{N},\pi) }\espi{\pi}{f(N\circ y^{*})}-H(\Ymap^{\#}\pi\,|\,\pi).
	\end{align*}

	It remains to prove that we can find $(\dot y_{n},\, n\ge 1)$, a sequence of
  elements of $\pred^{++}_{\infty,\pc}(\mathcal{N},\pi)$ such that
	\begin{align*}
	  \int_{\NN}f\dif \left( \Ymap_{n}^{\#}\pi \right) & \xrightarrow{n\to \infty}    \int_{\NN}f\dif \mu_{f} \\
	  H(\Ymap_{n}^{\#}\pi\,|\, \pi)                    & \xrightarrow{n\to \infty} H(\mu\, |
                                                       \, \pi)
	\end{align*}
	to conclude. This is the object of the next theorem.
\end{proof}

\begin{theorem}
  Let $\nu$ be a probability measure on $\NN$ absolutely continuous with respect
  to~$\pi$ and
  \begin{equation*}
    L=\frac{\D \nu}{\D \pi}\cdotp
  \end{equation*}
  Assume that $L\log L \in L^{1}(\pi)\text{ and } \log L \in L^{r}(\pi) \text{
    for some } r>1$.

  Then, there exists $(\dot y_{n},\ n\ge 1)$ a sequence of elements of
  $\pred^{++}_{\infty,\pc}(\mathcal{N},\pi)$ such that
  \begin{equation}\label{eq_hawkes_brouillon:42}
    L_{n}\log L_{n}\xrightarrow[n\to \infty]{L^{1}(\pi)} L\log L
  \end{equation}
  and
  \begin{equation}\label{eq_hawkes_brouillon:43}
    L_{n} \log L\xrightarrow[n\to \infty]{L^{1}(\pi)} L\log L
  \end{equation}
  where
  \begin{equation*}
    L_{n}=\frac{\D\, \Ymap_{n}^{\#} \pi}{\D \pi}\cdotp
  \end{equation*}
\end{theorem}
\begin{proof}
  We first show that we can suppose $L$ lower and upper bounded. Consider
  \begin{equation*}
    \Phi_{n}=\left( L\wedge n \right)\vee  \frac1n\cdotp
  \end{equation*}
  We have
  \begin{equation*}
    |\Phi_{n}|\le L+1,
  \end{equation*}
  hence by dominated convergence, $\Phi_{n}$ converges in $L^{1}(\pi)$ to $L$
  and in particular,
  \begin{equation*}
    \espi{\pi}{\Phi_{n}}\xrightarrow{n\to \infty} \espi{\pi}{L}=1.
  \end{equation*}
  Let
  \begin{equation*}
    L_{n}=\frac{\Phi_{n}}{\espi{\pi}{\Phi_{n}}}\cdotp
  \end{equation*}
  For any $\alpha\in (0,1)$, for $n$ sufficiently large,
  $\espi{\pi}{\Phi_{n}}\geq \alpha$. Moreover, for $x\ge 0$, we have
  \begin{equation*}
    |x\log(x)|\le \frac1{e}+\left| \frac{{x}}\alpha\log(\frac{{x}}\alpha) \right|\car_{x\ge \alpha}.
  \end{equation*}
  Hence,
  \begin{equation*}
    |L_{n}\log L_{n}|\le \frac1{e}+\left| \frac{L}{\alpha}\log \left( \frac{L}{\alpha} \right)  \right|.
  \end{equation*}
  By dominated convergence again, $L_{n}\log L_{n}$ converges to $L\log L$ in
  $L^{1}(\pi)$. Similarly,
  \begin{equation*}
    \left| L_{n}\log L \right|\leq \left| \frac{L}{\alpha}\log L \right|
  \end{equation*}
  and $L_{n}\log L$ converges to $L\log L$ in $L^{1}(\pi)$.

  Assume now that $L$ is lower and upper bounded by respectively $m$ and $M$. We
  know that there exists $\dot y\in \pred^{++}(\mathcal{N},\pi)$ such that
  \begin{equation*}
    L=\Lambda_{y}.
  \end{equation*}

  Set $$L_{n}=\Lambda_{1+(\dot y^{n}-1)}=\espi{\pi}{L\,|\, \mathcal{N}_{n}}.$$
  Since $L$ is bounded, it is clear that \eqref{eq_hawkes_brouillon:42} and
  \eqref{eq_hawkes_brouillon:43} hold. We can then assume that there exists
  $T>0$ such that $\dot y(N,s)=1$ for $s\ge T$.

  Moreover, from the Malliavin calculus for Poisson process, we know (see
  \cite{DecreusefondUpperboundsRubinstein2010}) that
  \begin{equation*}
    \dot y(N,s)=\frac{\espi{\pi}{D_{s}L\,|\, \mathcal{N}_{s}}}{\espi{\pi}{L\,|\,\mathcal{N}_{s}}}
  \end{equation*}
  where $D_{s}F(N)=F(N+\epsilon_{s})-F(N)$. We then have
  \begin{equation*}
    0\le \dot y(N,s)\le \frac{M}{m}\cdotp
  \end{equation*}
  Consider $\dot y_{n}=\dot y\vee n^{-1}$, it is straightforward that
  \eqref{eq_hawkes_brouillon:42} and \eqref{eq_hawkes_brouillon:43} hold.

  Finally, assume that $\dot y$ is lower and upper bounded on some interval
  $[0,T]$ and equal to $1$ above $T$. Set
  \begin{align*}
    \dot y_{n}(s) & =0 \text{ if } s\in [0,T/n)                                                \\
    \dot y_{n}(s) & =n\int_{(i-1)/n}^{i/n}\dot y(N,s)\dif s \text{ if } s\in [iT/n,\ (i+1)T/n)
  \end{align*}
  for $i\in \{1,\cdots, n-1\}$. We see that $\dot y_{n}$ belongs to
  $\pred^{++}_{\infty,\pc}(\mathcal{N},\pi)$. We know (see
  \cite{nualart_malliavin_1995}) that $\dot y_{n}$ converges in
  $L^{2}(\NN\times[0,T],\pi\otimes \D s)$ to $\dot y$. Moreover, it is easy to
  see that
  \begin{equation*}
    \sup_{n}\espi{\pi}{\Lambda_{y_{n}}^{p}}<\infty \text{ for any }p\ge 1.
  \end{equation*}
  Thus, \eqref{eq_hawkes_brouillon:42} and \eqref{eq_hawkes_brouillon:43} hold.
\end{proof}

\section{Weak and strong g-Hawkes processes}
\label{sec:weak-strong-hawkes}
This section does not directly utilize the prior results; rather, it is inspired
by the analogy drawn between addressing a volatility-$1$ Brownian stochastic
differential equation (SDE) and the formulation of a generalized Hawkes process.
In the context of SDEs, three distinct types of solutions are recognized: strong
solutions if for any given filtered probability space on which $B$ is built, we
can build a process $X$ which satisfies \eqref{eq_hawkes_brouillon:10}; weak
solutions if we have to specify the probability space; and martingale solutions
which require that the expression
\begin{equation*}
  X(t)-\int_{0}^{t}b\bigl(X(s),s\bigr)\dif s
\end{equation*}
is a local martingale of square bracket $(t\mapsto t)$. The interrelations among
these various types of solutions have been well-documented in the literature and
can be found in numerous textbooks, such as \cite{Stroock2006}. This culminates
in the Yamada-Watanabe theorem, which asserts that weak existence and strong
uniqueness together imply strong existence. In this work, we demonstrate that
analogous definitions of these different types of solutions can be established
for the construction of Hawkes processes, and we find a precise correspondence
to the Yamada-Watanabe theorem.

\def\F{\mathcal{F}}
\begin{definition}
  A filtered probability space is a triplet $(\Omega,\F,\P)$ where $\Omega$ is a
  space, equipped with a right-continuous filtration $\F$ and a probability
  $\P$.
\end{definition}
In what follows, we equip $\NN$ with the minimum filtration:
\begin{equation*}
  \mathcal{N}_{t}=\sigma\left\{\omega\bigl([0,s]\bigr),\, 0\le s\le t\right\}.
\end{equation*}
The process $y\, :\, \NN\to \R$ is supposed to be $\mathcal{N}$-predictable and
$\dot y$ is positive. We define the different notions of solution associated to
the generalized Hawkes problem associated to~$y$.
\begin{definition}[g-Hawkes problem]\label{def_hawkes_core:hawkes_problem}
  Consider $(\Omega,\mathcal{F},\P)$ a filtered probability space. By a solution
  of $\hawkes$, we mean a couple of processes $\ZZ=(Z,N)$ such that
  \begin{enumerate}
    \item \label{item_hawkes_core:6} With probability one, $Z$ and $N$ are point
          processes in the sense of
          Definition~\ref{def_hawkes_core:definition_point_processes},
    \item \label{item_hawkes_core:7} For any $t\ge 0$, the random variable
          $y(Z,t)$ is a $\F$ stopping time,
    \item \label{item_hawkes_core:8} The process
          \begin{math}
            ( N(t)-t,\ t\ge 0)
          \end{math}
          is a $\F$-local martingale,
    \item \label{item_hawkes_core:9} The processes $Z$ and $N$ satisfy,
          $\P$-a.s. for any $t\ge 0$,
          \begin{align}
            Z(t) & =N\Bigl(y(Z,t)\Bigr).\label{eq_hawkes_core:26}
          \end{align}
  \end{enumerate}
\end{definition}
\begin{remark}
  Note that item \ref{item_hawkes_core:6}. of
  Definition~\ref{def_hawkes_core:hawkes_problem} and \eqref{eq_hawkes_core:26}
  imply that
  \begin{equation}\label{eq_hawkes_core:6}
    y\Bigl(\bigl\{T_{1}(Z),\cdots,T_{q-1}(Z)\bigr\},\, T_{q}(Z)\Bigr)=T_{q}(N), \ \forall q\ge 1.
  \end{equation}
  However, \eqref{eq_hawkes_core:6} does not imply immediately
  \eqref{eq_hawkes_core:26} as it is not clear that $y(Z,t)$ is an $\mathcal{N}$
  stopping time and thus that the application $t\mapsto N\bigl(y(Z,t)\bigr)$ is
  measurable.
\end{remark}
\begin{definition}[Weak and strong solutions]
  If we must specify the probability space $(\Omega,\F,\P)$ then the solution is
  said to be weak.

  If we can find a solution for any filtered probability space on which we can
  construct a unit rate Poisson process, then the solution is said to be strong.
\end{definition}
\begin{definition}[Martingale problem]
  Let $(\Omega,\mathcal{F},\P)$ be given and $Z$ a point process.
  If the following conditions are satisfied:
  \begin{enumerate}[label=\roman*),resume]
    \item \label{item:8} $y^{*}(Z,t)$ is $\F$-adapted,
    \item \label{item:9} the process
          \begin{equation*}
            t\mapsto Z\bigl( y^{*}(Z,t) \bigr)-t
          \end{equation*}
          is a $\F$-local martingale,
  \end{enumerate}
  we then say that $(Z,\F)$ satisfies the g-Hawkes martingale problem, denoted
  by $\hawkes^{m}$.
\end{definition}
The first theorem is the exact analog of what happens for stochastic
differential equations driven by a Brownian motion.
\begin{theorem}[Equivalence of weak and martingale
  solutions]\label{thm:weakToMartingale}
  Let $\dot y \in \pred^{++}(\F)$. There exists a weak solution to $\hawkes$ if
  and only if there exists a solution to $\hawkes^{m}$.
\end{theorem}
\begin{proof}
  Assume that there exists $(\Omega,\F,\P)$ and $\Upsilon=(Z,N)$ a solution to
  $\hawkes$. We assumed that $\dot y>0,$ thus
  \begin{equation*}
    \bigl(y^{*}(Z,t)\le s\bigr)=\bigl(y(Z,s)\ge t\bigr)\in  \F_{t},
  \end{equation*}
  since $y(Z,t)$ is an $\mathcal{F}$-stopping time. Thus the process
  $(y^{*}(Z,t),\, t\ge 0)$ is $\F$-adapted. Moreover, as $y^{*}(Z,.)$ is an
  homeomorphism on $\R^{+}$, we have, $\P$-a.s.,
  \begin{equation*}
    Z\bigl( y^{*}(Z,t) \bigr)=N(t),\  \forall t\ge 0,
  \end{equation*}
  and then
  \begin{equation*}
    t\longmapsto Z\bigl( y^{*}(Z,t) \bigr)
  \end{equation*}
  is an $\F$-Poisson process of intensity~$1$, i.e. point \ref{item:9} holds and
  $(Z,\F)$ solves $\hawkes^{m}$.

  Conversely, if $(\Omega,\F,\P,Z) $ solves $\hawkes^{m}$, this means that $N$
  defined by
  \begin{equation*}\label{eq_hawkes_brouillon:17b}
    N(t):=Z\bigl( y^{*}(Z,t) \bigr)
  \end{equation*}
  is a unit rate Poisson process with respect to the filtration $\F$
  and taking the inverse of $y^{*}(Z,.)$, we have
  \begin{equation*}
    Z(t)=N\bigl(y(Z,t)\bigr),
  \end{equation*}
  thus~\eqref{eq_hawkes_core:26} is satisfied. Moreover, from \ref{item:8} we
  deduce that
  \begin{equation*}
    \bigl(y(Z,t)\le s\bigr)=\bigl(y^{*}(Z,s)\ge t\bigr)\in \F_{s}.
  \end{equation*}
  Hence Point \ref{item_hawkes_core:7} of
  Definition~\ref{def_hawkes_core:hawkes_problem} is satisfied and we can then
  say that $(Z,N)$ satisfies $\hawkes$ on $(\Omega,\F,\P)$ where $N$ is defined
  by~\eqref{eq_hawkes_brouillon:17b}.
\end{proof}
\begin{definition}[Pathwise uniqueness]
  We say that path-wise uniqueness holds for $\hawkes$ whenever for any two
  solutions $\ZZ=(Z,N)$ and $\ZZ'=(Z',N')$ defined on the same filtered
  probability space $(\Omega,\mathcal{F},\P)$, we have
  \begin{equation*}
    Z=Z',\ N=N'
  \end{equation*}
  up to $\P$-indistinguishability.
\end{definition}
\begin{definition}[Weak Uniqueness]
  We say that weak uniqueness holds for $\hawkes$ whenever for any two solutions
  $\ZZ=(Z,N)$ and $\ZZ'=(Z',N')$, possibly defined on different probability
  spaces, the law of $Z$ and $Z'$ on the space $(\NN,\mathcal{N}_{\infty})$ do
  coincide.
\end{definition}
For any $k\ge 1$, on $\NN^{k}$, we consider the $\sigma$-field
$\mathcal{N}^{\otimes (k)}$ defined by
\begin{equation*}
  \mathcal{N}^{\otimes (k)}=\otimes_{i=1}^{k}\mathcal{N}^{i}
  \text{ where }
  \mathcal{N}^{i}
  =\sigma\left(\eta_{i}(s),\, s\ge 0\right), \text{ for } i=1,\cdots,k.
\end{equation*}
For any $t\ge 0$, we introduce the sub-$\sigma$ field
\begin{equation*}
  \mathcal{N}^{\otimes (k)}_{t}=\otimes_{i=1}^{k}\mathcal{N}^{i}_{t}.
\end{equation*}
For any solution $\ZZ$ of $\hawkes$, let $\mu$ be the law of $\ZZ$ on
$(\NN^{2},\, \mathcal{N}^{\otimes (2)})$, i.e. $\mu=\Upsilon^{\#}\P$. Let
$\mu_{\eta_{2}}(\dif \eta_{1})$ be the regular conditional distribution of
$\mu(\dif \eta_{1},\dif \eta_{2})$ given $\eta_{2}$:
\begin{enumerate}
  \item For each $\eta_{2}$, $\mu_{\eta_{2}}(\dif \eta_{1})$ is a probability
        measure on $(\NN,\mathcal{N}^{1})$,
  \item For each $B\in \mathcal{N}^{1}$, $\mu_{\eta_{2}}(B)$ is
        $\mathcal{N}^{2}$-measurable in $\eta_{2}$,
  \item For any $B\in \mathcal{N}^{1},\, B'\in \mathcal{N}^{2}$,
        \begin{equation*}
          \mu(B\times B')=\int_{B'}\mu_{\eta_{2}}(B)\dif \pi(\eta_{2}).
        \end{equation*}
\end{enumerate}
Let $\mu_{\eta_{2}}^{t}$ be the regular conditional distribution of $\mu$ given
$\mathcal{N}_{t}$:
\begin{enumerate}
  \item For each $\eta_{2}\in \NN$, $\mu_{\eta_{2}}^{t}$ is a measure on $\NN$,
  \item for each $B\in \mathcal{N}$, the random variable $\mu_{\eta_{2}}^{t}(B)$
        is $\mathcal{N}_{t}\ /\ \mathcal{B}(\R^{+})$-measurable,
  \item for any $B'\in \mathcal{N}_{t}$,
        \begin{equation*}
          \mu(B\times B')=\int_{B'}\mu_{\eta_{2}}^{t}(B)\dif \pi(\eta_{2}).
        \end{equation*}
\end{enumerate}
\begin{lemma}
  \label{lem_hawkes_core:mesurabilite_Z}
  For $B\in \mathcal{N}_{y^{*}(\eta_{1},t)}^{1}$, the map
  $\eta_{2}\mapsto \mu_{\eta_{2}}(B)$ is
  $\mathcal{N}_{t}^{2}\ /\ \mathcal{B}(\R^{+})$-measurable.
\end{lemma}

\begin{proof}
  We first prove that for $B\in \mathcal{N}^{1}_{y^{*}(\eta_{1},t)}$, there
  exists $\theta$ measurable, such that
  \begin{equation}
    \label{eq_hawkes_core:34}
    \indic_{B}(\eta_{1})=\theta(\eta_{2}^{t}), \ \mu\text{-a.s.}
  \end{equation}
  or equivalently, $\indic_{B}(\eta_{1})\in \mathcal{N}^{2}_{t}.$ According to
  \cite{JacodCalculstochastiqueproblemes1979}, we know that
  $\mathcal{N}^{1}_{y^{*}(\eta_{1},t)}$ is generated by sets of the form
  \begin{equation}\label{eq_hawkes_core:35}
    \bigcap_{j=0}^{q}\Bigl(T_{j}(\eta_{1})\le a_{j}\Bigr)\cap \Bigl(T_{q}(\eta_{1})<y^{*}(\eta_{1},t)\le T_{q+1}(\eta_{1})\Bigr)
  \end{equation}
  for some $a_{j}\in \R^{+}$ and some $q\in \N$. Recall that
  \eqref{eq_hawkes_core:26} implies that
  \begin{equation*}
    \mu\Bigl(\eta_{1}(t)=\eta_{2}\bigl(y(\eta_{1},\, t)\bigr), \ \forall t\ge 0\Bigr)=1.
  \end{equation*}
  This means that
  \begin{equation*}
    y \bigl(\eta_{1},T_{k}(\eta_{1})\bigr)=T_{k}(\eta_{2})
    \Longleftrightarrow T_{k}(\eta_{1})=y^{*}\bigl(\eta_{1},T_{k}(\eta_{2})\bigr).
  \end{equation*}
  Hence for a set $B$ of the form \eqref{eq_hawkes_core:35}, we have
  \begin{align*}
    B & =\bigcap_{j=0}^{q}\Bigl(y^{*}\bigl(\eta_{1},T_{j}(\eta_{2})\bigr)\le a_{j}\Bigr)\cap \Bigl(T_{q}(\eta_{2})<t\le T_{q+1}(\eta_{2})\Bigr) \\
      & =\bigcap_{j=0}^{q}\Bigl(T_{j}(\eta_{2})\le y(\eta_{1},a_{j})\Bigr)\cap \Bigl(T_{q}(\eta_{2})<t\le T_{q+1}(\eta_{2})\Bigr)
  \end{align*}
  By definition of the g-Hawkes problem, $y(\eta_{1},a)$ is a
  $\mathcal{N}^{2}$-stopping time hence for any $j\le q$,
  \begin{equation*}
    \Bigl(T_{j}(\eta_{2})\le y(\eta_{1},a_{j})\Bigr)\in \mathcal{N}^{2}_{T_{j}(\eta_{2})\wedge y(\eta_{1},a_{j})}\subset \mathcal{N}^{2}_{T_{j}(\eta_{2})}\subset \mathcal{N}^{2}_{T_{q}(\eta_{2})}.
  \end{equation*}
  Hence $ B\in \mathcal{N}^{2}_{t}. $

  Let $B\in \mathcal{N}_{y^{*}(\eta_{1},t)}^{1}$. For the second step of the
  proof, we have to prove that for $F\,:\, \NN\to \R$ measurable and bounded
  \begin{equation*}
    \int_{\NN\times \NN}F(\eta_{2})\indic_{B}(\eta_{1})\ \mu(\dif \eta_{1},\, \dif \eta_{2})= \int_{\NN\times \NN}F(\eta_{2})\mu_{\eta_{2}}^{t}(B)\ \mu(\dif \eta_{1},\, \dif \eta_{2}).
  \end{equation*}
  Since $\eta_{2}$ is a Poisson process, for any $F$ bounded, the martingale
  representation theorem valid for Poisson processes says that there exists
  $u_{F}$ which is $\mathcal{N}$-adapted such that
  \begin{equation*}
    \esp{\int_{0}^{\infty}u_{F}(\eta_{2},\,s)^{2}\dif s}<\infty
  \end{equation*}
  and
  \begin{equation*}
    F(\eta_{2})=\esp{F}+\int_{0}^{\infty}u_{F}(\eta_{2},\,s)\dif \tilde \eta_{2}(s)
  \end{equation*}
  where $\tilde \eta_{2}(t)=\eta_{2}(t)-t.$ The process
  \begin{equation*}
    (\eta_{2},t)\longmapsto \int_{0}^{t}u_{F}(\eta_{2},\,s)\dif \tilde\eta_{2}(s)
  \end{equation*}
  is a square integrable martingale with respect to the filtration $\mathcal{N}$
  and thus it is also a square martingale with respect to the filtration
  $\mathcal{N}^{\otimes (2)}$. Now then we have
  \begin{multline*}
    \int_{\NN\times \NN}F(\eta_{2})\indic_{B}(\eta_{1})\ \mu(\dif \eta_{1},\, \dif \eta_{2})= \esp{F}\int_{\NN\times \NN}\indic_{B}(\eta_{1})\ \mu(\dif \eta_{1},\, \dif \eta_{2})\\
    + \int_{\NN\times \NN}\left(\int_{t}^{\infty}u_{F}(\eta_{2},s)\dif \tilde\eta_{2}(s)\right)\indic_{B}(\eta_{1})\ \mu(\dif \eta_{1},\, \dif \eta_{2})\\+ \int_{\NN\times \NN}\left(\int_{0}^{t}u_{F}(\eta,s)\dif \tilde\eta(s)\right)\indic_{B}(\eta_{1})\ \mu(\dif \eta_{1},\, \dif \eta_{2}).
  \end{multline*}
  By the first part of the proof, $\indic_{B}(\eta_{1})=\theta(\eta_{2}^{t})$
  $\mu$-a.s. Hence, by the martingale property of the stochastic integral, the
  median term is null. We thus get
  \begin{align*}
    \int_{\NN\times \NN}F(\eta)\indic_{B}(z)\ \mu(\dif z,\, \dif \eta) & =\int_{\NN} \left[\esp{F}+\left(\int_{0}^{t}u_{F}(\eta,s)\dif \tilde\eta(s)\right)\right] \ \mu_{\eta}^{t}(B)\dif \pi(\eta) \\
                                                                       & =\int_{\NN} F(\eta)\ \mu_{\eta}^{t}(B)\dif \pi(\eta)
  \end{align*}
  by the same kind of reasoning.

\end{proof}
\begin{theorem}[Yamada-Watanabe]\label{thm_hawkes_core:YW}
  With the same notations as above. Consider the $\hawkes$ problem as in
  Definition~\ref{def_hawkes_core:hawkes_problem}. The following two properties
  hold:
  \begin{enumerate}
    \item Pathwise uniqueness implies weak uniqueness.
    \item Moreover, if there exists a solution $\ZZ=(Z,N)$ of $\hawkes$ on some
          $(\Omega,\mathcal{F},\P)$ and pathwise uniqueness holds then there
          exists $F\,:\, \NN\to \NN$ such that $Z=F(N)$ $\P$-a.s. Furthermore,
          the map $F$ is $\mathcal{N}_{t}\, /\, \mathcal{Z}_{y^{*}(Z,t)}$
          measurable, i.e. for any $t\ge 0$,
          \begin{equation*}
            \mathcal{Z}_{y ^{*}(Z,t)}=\sigma\left\{F\bigl(N(s)\bigr),\, s\le t\right\}.
          \end{equation*}
  \end{enumerate}
\end{theorem}
\begin{proof}
  Let $\ZZ=(Z,N)$ and $\ZZ'=(Z',N')$ two solutions of $\hawkes$, which are
  possibly defined on different probability spaces, and $\mu$ and $\mu'$ there
  respective distribution on $(\NN^{2};\, \mathcal{N}_{\infty}^{\otimes(2)})$.
  On $(\NN^{3},\, \mathcal{N}_{\infty}^{\otimes(3)})$, we define the probability
  measure
  \begin{equation*}
    \nu(\dif \eta_{1},\dif \eta_{2},\dif \eta_{3})=\mu_{{\eta_3}}(\dif \eta_{1})\mu'_{{\eta_3}}(\dif \eta_{2})\dif \pi(\eta_{3}).
  \end{equation*}
  We are going to prove that $\eta_{{3}}$ is a $\nu$ Poisson process. For, for
  $i=1,2,3$, let $F_{i}$ a bounded function,
  $\mathcal{N}_{y^{*}(\eta_{i},s)}^{i}\, /\, \mathcal{B}(\R)$ measurable for
  $i=1,2$ and $\mathcal{N}^{3}_{s}\,/\, \mathcal{B}(\R)$ measurable for $i=3$.
  We have
  \begin{multline}\label{eq_hawkes_core:27}
    \int_{\NN^{3}}F_{1}(\eta_{1})\ F_{2}(\eta_{2})\ F_{3}(\eta_{3})\ \bigl(\tilde \eta_{3}(t)-\tilde \eta_{3}(s)\bigr)\nu(\dif \eta_{1},\dif \eta_{2},\dif \eta_{3})\\=
    \int_{\NN} \left(\int_{\NN}F_{1}(\eta_{1})\mu_{\eta_{3}}(\dif \eta_{1})\right) \left(\int_{\NN }F_{2}(\eta_{2})\mu'_{\eta_{3}}(\dif \eta_{2})\right)\bigl(\tilde \eta_{3}(t)-\tilde \eta_{3}(s)\bigr)\dif \pi(\eta_{3}).
  \end{multline}
  {where  $\tilde{\eta}_{3}(t)= \eta_3(t) -t.$}
  According to Lemma \ref{lem_hawkes_core:mesurabilite_Z}, the random variables
  \begin{equation*}
    \int_{\NN}F_{1}(\eta_{1})\mu_{\eta_{3}}(\dif \eta_{1}) \text{ and }
    \int_{\NN}F_{2}(\eta_{2})\mu'_{\eta_{3}}(\dif \eta_{2})
  \end{equation*}
  are $\mathcal{N}^{3}_{s}$ measurable
  hence $\mathcal{N}^{\otimes(3)}_{s}$-measurable. It follows that the
  right-hand-side of \eqref{eq_hawkes_core:27} is equal to zero and that
  $\eta_{3}$ is a
  $\Bigr(\nu,\, \bigl(\mathcal{N}^{1}_{y^{*}(\eta_{1},t)}\otimes \mathcal{N}^{2}_{y^{*}(\eta_{2},t)}\otimes \mathcal{N}^{3}_{t},\ t\ge 0\bigr)\Bigl)$
  unit rate Poisson process.

  We thus have two solutions $(\eta_{1},\eta_{3})$ and $(\eta_{2},\eta_{3})$ on
  the same probability space.
  The pathwise uniqueness then implies that $\eta_{1}=\eta_{2}$, $\nu$-a.s.
  Thus, $\mu(\dif \eta_{1},\dif \eta_{3})=\mu'(\dif \eta_{2},\dif \eta_{3})$ and
  the uniqueness in law holds.

  In view of Lemma~\ref{lem:EgaliteFiltration} and Corollary
  \ref{lem_hawkes_core:egalité_tribu_hawkes}, the Doob Lemma says that there
  exist two functions $F_{i},i=1,2$ respectively measurable from
  $(\NN, \mathcal{N}^{3}_{\infty})$ to
  $(\NN,\mathcal{N}^{i}_{y^{*}(\eta_{i},\infty)})=(\NN,\mathcal{N}^{i}_{\infty}), i=1,2$
  such that $\eta_{1}^{y^{*}(\eta_{1},t)}=F_{1}(\eta_{3}^{t})$ and
  $\eta_{2}^{y^{*}(\eta_{2},t)}=F_{2}(\eta_{3}^{t})$ for any $t\ge 0$.
  Furthermore, for $\pi$-almost all $\eta_{3}$,
  \begin{equation*}
    \mu_{\eta_{3}}\otimes \mu_{\eta_{3}}\left(\eta_{1}=\eta_{2}\right)=1
  \end{equation*}
  and this implies that $F_{1}=F_{2}$, i.e. there exists $F$ such that
  $\eta_{1}=\eta_{2}=F(\eta_{3})$.
\end{proof}

\begin{theorem}\label{thm_hawkes_brouillon:weakExistence}
  Let $(\Omega,\F,\P)$ be a probability space and $N$ a unit Poisson process
  such that $\F=\mathcal{N}$. We denote by $\pi$ the law of $N$. For any
  $\dot y\in \pred^{++}_{2}(\mathcal{N},\pi_{y})$, there exists a (weak)
  solution to $\hawkes$.
\end{theorem}
\begin{proof}
  In view of Theorem \ref{thm_hawkes_brouillon:absolueContinuite}, $\pi_{y}$ is
  locally absolutely continuous with respect to~$\pi$ and the quasi-invariance
  theorem says that the process $Y$ defined by
  \begin{equation*}
    Y(t):= N\Bigl( y^{*}(N,t) \Bigr),
  \end{equation*}
  is a $(\pi_{y},\F^{y^{*}})$ unit Poisson process. Thus, we have
  \begin{equation*}
    N(t)=Y\bigl( y(N,t) \bigr)
  \end{equation*}
  and $y(N,t)$ is a $\F^{y^{*}}$-stopping time, hence $(N,Y,\F^{y^{*}})$ solves
  $\hawkes^{w}$.
\end{proof}

\begin{theorem}\label{thm_hawkes_brouillon:strongExistenceClassicalHawkes}
  Consider $(\Omega,\F,\P)$ be a filtered probability space and $N$ a unit
  Poisson process. Let $\dot y\in \pred^{++}_{2}(\mathcal{N},\pi_{y})$ such that
  for any $\omega,\eta\in \NN$, for any $t\ge 0$
  \begin{equation}
    \label{eq:LipschitzCondition}
    |\dot y(\omega,t)-\dot y(\eta,t)| \le \int_{0}^{t } |\phi(t-s)| \dif |\omega-\eta|(s)
  \end{equation}
  where $|\omega-\eta|$ is the point process defined by
  \begin{equation*}
    |\omega-\eta|(t)=|\omega(t)-\eta(t)|
  \end{equation*}
  and $\phi$ is such that
  \begin{equation*}
    \int_{0}^{\infty}\phi(s)\dif s<1.
  \end{equation*}
  Then, there exists a unique solution to $\hawkes^{s}$.
\end{theorem}
\begin{remark}
  The condition \eqref{eq:LipschitzCondition} is the condition used in
  \cite{MR0166826,massoulie96} to ensure existence and uniqueness of the
  solution of the (weak) g-Hawkes problem.
\end{remark}
\begin{proof}
  The weak existence is guaranteed by
  Theorem~\ref{thm_hawkes_brouillon:weakExistence}. It remains to show the
  strong uniqueness and to conclude with Theorem~\ref{thm_hawkes_core:YW}. If
  $X$ and $Y$ are two solutions of $\hawkes$ on the same probability space, we
  have
  \begin{equation*}
    X(t)=N\bigl(y(X,t)\bigr) \text{ and } Y(t)=N\bigl(y(Y,t)\bigr).
  \end{equation*}
  Since they both are point processes, we can speak of their jump times
  $\bigl(T_{q}(X),\, q\ge 1\bigr)$ and $\bigl(T_{q}(Y),\, q\ge 1\bigr)$. We have
  \begin{equation*}
    y(\emptyset,\, T_{1}(X))=T_{1}(N) \text{ and } y(\emptyset,\, T_{1}(Y))=T_{1}(N).
  \end{equation*}
  Since $y(N,.)$ is an homeomorphism, $T_{1}(X)=T_{1}(Y)$. Suppose proved that
  \begin{equation*}
    T_{j}(X)=T_{j}(Y) \text{ for }j=1,\cdots,q-1.
  \end{equation*}
  Since $y$ is predictable
  \begin{align*}
    y\bigl(\sum_{j=1}^{q-1}\epsilon_{T_{j}(X)},\, T_{q}(X)\bigr)=T_{q}(N)= y\bigl(\sum_{j=1}^{q-1}\epsilon_{T_{j}(X)},\, T_{q}(Y)\bigr).
  \end{align*}
  Hence $T_{q}(X)=T_{q}(Y)$ and then $X=Y$. Thus the strong uniqueness holds.
\end{proof}

\begin{remark}
  Actually, we have used in this proof the construction inherited from Algorithm
  7.4.III of \cite{Daley2003} to simulate a g-Hawkes process. In view of
  \eqref{eq_hawkes_brouillon:7} and \eqref{eq_hawkes_brouillon:sautsDeY},
  $T_{1}(Z)$ is the solution of the equation
  \begin{equation*}
    y\left(\emptyset, T_{1}(Z) \right)=T_{1}(N).
  \end{equation*}
  In turn, $T_{2}(Z)$ solves
  \begin{equation*}
    y\left( \epsilon_{T_{1}(Z)},\, T_{2}(Z) \right)=T_{2}(N).
  \end{equation*}
  More generally, $T_{k}(Z)$ is given by
  \begin{equation*}
    y\left( \sum_{j=1}^{k-1}\epsilon_{T_{j}(Z)},\, T_{k}(Z) \right)=T_{k}(N).
  \end{equation*}
  The usual algorithm which is based on the thinning of a Poisson measure gives
  the sample path of $Z$ up to a given time and suffers from the rejection of a
  possibly large number of points. This algorithm yields the number of jumps we
  desire and may suffer from numerical instability in the root findings.
\end{remark}

We can now retrieve the classical existence result for classical Hawkes
processes \cite{hawkes74,massoulie96}.
\begin{corollary}
  Let $\dot y $ be a predictable process which satisfies the following
  hypothesis
  \begin{equation*}
    \dot y(N,t)=\varphi\left( \alpha+\int_{0}^{t} h(t-s)\dif N(s) \right)
  \end{equation*}
  where $\alpha>0$, $\varphi$ is a Lebesgue a.s. positive Lipschitz function and
  $h$ is a non-negative measurable function such that
  \begin{equation}\label{eq_hawkes_brouillon:contration}
    \|\varphi\|_{\text{Lip}}   \,\int_{0}^{\infty} h(t) \dif t <1.
  \end{equation}
  Then, there exists a unique strong solution to $\hawkes$.
\end{corollary}
\begin{proof}
  Let $T>0$ and
  \begin{equation*}
    \dot y_{T}(t)=
    \begin{cases}
      \dot y(N,s) & \text{ if }s\le T \\
      1           & \text{ if } s>T.
    \end{cases}
  \end{equation*}
  We first prove that there exists a solution to $\mathfrak H_{y_{T}}$. We have
  to prove that
  \begin{equation}\label{eq_hawkes_core:10}
    \pi_{y}\left( \int_{0}^{T}\left( 1-\sqrt{\dot y(N,s)} \right)^{2}\dif s <\infty\right)=1.
  \end{equation}
  It is easily seen that for all $x\ge -1$, we have
  \begin{equation*}
    \left(  \sqrt{1+x}-1 \right)^{2}\le  x^{2}\car_{\{|x|\le 1\}}+x\car_{\{|x|> 1\}} .
  \end{equation*}
  Since $\varphi$ is supposed to be Lipschitz continuous, we get
  \begin{multline*}
    \int_{0}^{T}\left( 1-\sqrt{\dot y(N,s)} \right)^{2}\dif s  \le T+\int_{0}^{T}\left|  \dot y(N,s)-1\right|\dif s                                                                          \\
    \le (1+|\varphi(0)-1|+\alpha\|\varphi\|_{\text{Lip}})T+\|\varphi\|_{\text{Lip}}\int_{0}^{T}\int_{0}^{s}h(s-u)\dif N(u)\dif s.
  \end{multline*}
  Thus, proving \eqref{eq_hawkes_core:10} amounts to prove that
  \begin{equation*}
    \pi_{y}\left(\int_{0}^{T} \int_{0}^{s}h(s-u)\dif N(u)\dif s<\infty  \right)=1.
  \end{equation*}
  Denote by $H$ the first quadrature of $h$:
  \begin{equation*}
    H(t)=\int_{0}^{t}h(s)\dif s.
  \end{equation*}
  By Fubini's Theorem, we have
  \begin{equation*}
    \int_{0}^{T} \int_{0}^{s}h(s-u)\dif N(u)\dif s=\int_{0}^{T}H(T-s)\dif N(s)\le H(T)N(T).
  \end{equation*}
  The same argument shows that
  \begin{equation*}
    \espi{\pi_{y}}{N(T)}\le cT+\|\varphi\|_{\text{Lip}}H(T)\,\espi{\pi_{y}}{N(T)}.
  \end{equation*}
  In view of \eqref{eq_hawkes_brouillon:contration}, this induces that
  $\espi{\pi_{y}}{N(T)}$ is finite and thus that \eqref{eq_hawkes_core:10}
  holds. It is clear that $\dot y$ so defined satisfies
  \eqref{eq:LipschitzCondition} and then, Theorem
  \ref{thm_hawkes_brouillon:strongExistenceClassicalHawkes} ensures the
  existence of the Hawkes process associated to $y$ on $[0,T]$ for any $T>0$. We
  denote this process by $Y_{T}$. Since $y$ is predictable, for all
  $r\le T\le S$,
  \begin{align*}
    Y_{T}(r) & =N\left( y_{T}(Y_{T},r) \right)  \\
             & =N\left( y_{S}(Y_{T},r) \right).
  \end{align*}
  By uniqueness, this means that $Y_{T}$ and $Y_{S}$ coincide on $[0,T]$. We can
  then define
  \begin{equation*}
    Y(r)=Y_{r}(r),
  \end{equation*}
  which satisfies
  \begin{equation*}
    Y(r)=N\left( y_{r}(Y(r),r) \right)=N\left( y(Y(r),r) \right)
  \end{equation*}
  by the very definition of $\dot y_{r}$. We have thus proved the existence of a
  strong solution to the Hawkes problem associated to $y$. The uniqueness comes
  from the Lipschitz condition \eqref{eq:LipschitzCondition} as in the proof of
  Theorem \ref{thm_hawkes_brouillon:strongExistenceClassicalHawkes}.
\end{proof}

\bibliography{hawkes_biblio} \bibliographystyle{amsplain}
\end{document}